\documentclass[11pt, reqno]{amsart}
\usepackage[bookmarksnumbered, plainpages]{hyperref}

\usepackage{amsmath,amssymb,graphicx,mathabx,accents}
\usepackage{enumerate,mdwlist}
\usepackage{enumitem}
\usepackage{esint}
\usepackage[symbol]{footmisc}
\usepackage{mathtools}
\usepackage{mathrsfs}

\usepackage{tikz}

\numberwithin{equation}{section}

\usepackage{amsthm}

\usepackage{verbatim}

\usepackage{nag}

\usepackage{tikz-cd} 

\makeatletter
\DeclareRobustCommand{\hvec}[1]{{\mathpalette\hvec@{#1}}}
\newcommand{\hvec@}[2]{%
  \vbox{\offinterlineskip
    \ialign{%
      \hfil##\hfil\cr
      $\m@th#1{}_{\rightharpoonup}$\kern-\scriptspace\cr
      $\m@th#1#2$\cr
    }%
  }%
}
\makeatother

\DeclareMathOperator{\RR}{\mathbb{R}}
\DeclareMathOperator{\ZZ}{\mathbb{Z}}
\DeclareMathOperator{\QQ}{\mathbb{Q}}
\DeclareMathOperator{\TT}{\mathbb{T}}
\DeclareMathOperator{\CC}{\mathbb{C}}
\DeclareMathOperator{\NN}{\mathbb{N}}
\DeclareMathOperator{\HH}{\mathbb{H}}

\DeclareMathOperator{\supp}{\text{supp}}

\newtheorem{theorem}{Theorem}
\newtheorem{lemma}[theorem]{Lemma}
\newtheorem{corollary}[theorem]{Corollary}
\newtheorem{prop}[theorem]{Proposition}

\newtheorem{remark}[theorem]{Remark}

\DeclareMathOperator{\PP}{\mathbb{P}}

\begin{document}

\centerline{}

\centerline{}

\title[Multipliers For Spherical Harmonic Expansions]{Multipliers For Spherical Harmonic Expansions} 
\author[Jacob Denson]{Jacob Denson$^*$}
\address{$^{*}$ University of Madison Wisconsin, Madison, WI, jcdenson@wisc.edu}
\subjclass[2010]{Primary 58J40; Secondary 58C40.}
\keywords{Functions of the Laplacian, Radial Fourier Multipliers, Endpoint Estimates, Local Smoothing for the Wave Equation, Transference Principles, Finsler Manifolds}

\begin{abstract}
    For any bounded, regulated function $m: [0,\infty) \to \CC$, consider the family of operators $\{ T_R \}$ on the sphere $S^d$ such that $T_R f = m(k/R) f$ for any spherical harmonic $f$ of degree $k$. We completely characterize the compactly supported functions $m$ for which the operators $\{ T_R \}$ are uniformly bounded on $L^p(S^d)$, in the range $1/(d-1) < |1/p - 1/2| < 1/2$. We obtain analogous results in the more general setting of multiplier operators for eigenfunction expansions of an elliptic pseudodifferential operator $P$ on a compact manifold $M$, under curvature assumptions on the principal symbol of $P$, and assuming the eigenvalues of $P$ are contained in an arithmetic progression. One consequence of our result are new transference principles controlling the $L^p$ boundedness of the multiplier operators associated with a function $m$, in terms of the $L^p$ operator norm of the radial Fourier multiplier operator with symbol $m(|\cdot|): \RR^d \to \CC$. In order to prove these results, we obtain new quasi-orthogonality estimates for averages of solutions to the half-wave equation $\partial_t - i P = 0$, via a connection between pseudodifferential operators satisfying an appropriate curvature condition and Finsler geometry.

\end{abstract} \maketitle

\section{Introduction}

Consider a \emph{zonal convolution operator}, a bounded operator $T: L^2(S^d) \to L^2(S^d)$ on the sphere invariant under rotations. 
Then $T$ is diagonalized by \emph{spherical harmonics}; there exists a function $m: \NN \to \CC$, the \emph{symbol} of the operator, such that
\begin{equation} \label{equationZonalExpansion}
    Tf = \sum\nolimits_{k = 0}^\infty m(k) f_k \quad\text{for all functions $f \in L^2(S^d)$},
\end{equation}
where $f = \sum_{k = 0}^\infty f_k$ is the spherical harmonic expansion of $f$. Conversely, any bounded function $m: \NN \to \CC$ defines such an operator. In this paper, we study $L^p$ estimates for such multiplier operators. For a function $m: [0,\infty) \to \CC$ and $R > 0$, we define
\begin{equation}
    T_Rf = \sum\nolimits_{k = 0}^\infty m \big( k/R \big) f_k.
\end{equation}
For a limited range of $p \in [1,\infty]$, we completely characterize all compactly supported functions $m$ for which the operators $\{ T_R \}$ are uniformly bounded on $L^p(S^d)$. Such operators naturally occur in the study of eigenfunction expansions, and we obtain new results in this setting. More generally, we consider \emph{spectral multiplier operators} on a compact manifold $M$, operators diagonalized by eigenfunctions of an elliptic pseudodifferential operator $P$ on $M$. Under the assumption that the eigenvalues of $P$ are contained in an arithmetic progression, and that the principal symbol of $P$ satisfies an appropriate curvature assumption, we obtain an analogous characterization of operators whose dilates are uniformly bounded in $L^p$. As a result, we obtain the first \emph{transference principle} taking $L^p$ estimates on Fourier multipliers on Euclidean space and giving $L^p$ estimates for spectral multipliers on a compact manifold.


{\bf Acknowledgements.} I would like to thank my advisor Andreas Seeger for innumerable conversations and insights which proved invaluable in producing the results of this paper. This research was supported by NSF grants DMS-2037851, DMS-2054220, and DMS-2348797.

\subsection{Statement of Main Results} Recall a function $m: [0,\infty) \to \CC$ is \emph{regulated} if
\begin{equation}
    m(\lambda_0) = \lim_{\delta \to 0} \fint_{|\lambda - \lambda_0| \leq \delta} m(\lambda)\; d\lambda \quad\text{for all $\lambda_0 \in [0,\infty)$}.
\end{equation}
The main result of this paper is a complete characterization of the bounded, regulated functions $m$ such that the multiplier operators $\{ T_R \}$ are uniformly bounded on $L^p(S^d)$, for a limited range of $L^p$ spaces. For a function $m: [0,\infty) \to \CC$, define   
\begin{equation} \label{definitonOfCp}
    C_p(m) = \left( \int_0^\infty \Big[ \langle t \rangle^{\alpha(p)} |\widehat{m}(t)| \Big]^p\; dt \right)^{1/p}
\end{equation}
where $\alpha(p) = (d-1)(1/p - 1/2)$, and $\widehat{m}(t) = \int_0^\infty m(\lambda) \cos(2 \pi \lambda t) d\lambda$ is the cosine transform of the function $m$.

\begin{theorem} \label{zonalconvolutionapplication}
    Suppose $1 < p < 2(d-1)/(d+1)$. For any bounded, regulated function $m$ with compact support in $(0,\infty)$, define operators $T_R$ on $S^d$ by setting
    \begin{equation}
        T_R f = \sum\nolimits_{k = 0}^\infty m(k/R) f_k,
    \end{equation}
    where $f_k$ is the projection of $f$ onto the space of degree $k$ spherical harmonics. Then,
    \begin{equation} \label{TRBounds}
        \sup\nolimits_{R > 0} \| T_R \|_{L^p(S^d) \to L^p(S^d)} \sim C_p(m),
    \end{equation}
    with implicit constants depending on the support of $m$.
\end{theorem}

\begin{remark} \label{RemarkHolder}
\normalfont
   The Hausdorff-Young inequality and Sobolev embedding theorem implies that any regulated function $m$ with $C_p(m) < \infty$ for $1 < p < 2d/(d+1)$ must be H\"{o}lder continuous of order $d(1/p - 1/2) - 1/2$. Thus one can think of $C_p(m)$ as a measure of the smoothness of the function $m$, and then Theorem \ref{CpVersionOfTheorem} aligns with classical results in harmonic analysis linking boundedness of multiplier operators and the smoothness of their symbol.


\end{remark}

\begin{remark}
\normalfont
By duality, for any function $m: [0,\infty) \to \CC$, and for $1/p + 1/q = 1$,
\begin{equation}
    \| T_R \|_{L^p \to L^p} = \| T_R \|_{L^q \to L^q}.
\end{equation}
It thus follows from Theorem \ref{zonalconvolutionapplication} that for $2(d-1)/(d-3) < p < \infty$, and any regulated function $m: [0,\infty) \to \CC$ supported on a compact subset of $(0,\infty)$,
\begin{equation}
    \sup\nolimits_R \| m_R(P) \|_{L^p(S^d) \to L^p(S^d)} \sim C_q(m).
\end{equation} 
We state our results for $1 < p < 2(d-1)/(d+1)$ since the geometric arguments we use in our proofs apply more naturally in the range $1 \leq p \leq 2$ than the range $2 \leq p \leq \infty$.
\end{remark}

\begin{remark} \label{ABRemark}
\normalfont
To get more explicit control over the implicit constants in Theorem \ref{CpVersionOfTheorem}, we can perform a dyadic decomposition. Given $m$ supported on an interval $(A,B)$, consider a decomposition $m_j(\lambda) = \chi(\lambda) m(2^j \lambda)$ for some fixed $\chi \in C_c^\infty(0,\infty)$ with $\sum \chi(\lambda / 2^j) = 1$. Applying Theorem \ref{CpVersionOfTheorem} and the triangle inequality, we see that
\begin{equation} \label{ABVersionOfTheorem}
    \sup\nolimits_R \| m_R(P) \|_{L^p(S^d) \to L^p(S^d)} \lesssim |\log(B/A)| \sup\nolimits_{j \in \ZZ} C_p(m_j).
\end{equation}
%
%
%
%
%
In a follow-up paper, we plan to use methods of atomic decompositions to remove the dependence of this bound on $A$ and $B$, i.e. proving that for any regulated $m$,
\begin{equation} \label{FutureEquation}
    \sup\nolimits_R \| m_R(P) \|_{L^p(S^d) \to L^p(S^d)} \sim \sup\nolimits_{j \in \ZZ} C_p(m_j).
\end{equation}
One consequence would be a characterization of $L^p$ bounded spectral multipliers as in Theorem \ref{zonalconvolutionapplication}, but for multipliers that are not necessarily compactly supported.
\end{remark}

In order to prove Theorem \ref{zonalconvolutionapplication}, it is most natural to work in the language of spectral multipliers. Let $M$ be a compact manifold equipped with some volume density, and let $P$ be a classical elliptic pseudodifferential operator of order one on $M$, formally self-adjoint with respect to the volume density and with non-negative principal symbol $p: T^*M \to [0,\infty)$. Then there exists a discrete set $\Lambda_P \subset [0,\infty)$ and a decomposition $L^2(M) = \bigoplus_{\Lambda_P} \mathcal{V}_\lambda$, where $Pf = \lambda f$ for $f \in \mathcal{V}_\lambda$. Any function $f \in L^2(M)$ has an orthogonal decomposition $f = \sum\nolimits_{\lambda \in \Lambda_P} f_\lambda$, with $f_\lambda \in \mathcal{V}_\lambda$, and for any bounded function $m: \Lambda_P \to \CC$, we define
\begin{equation}
    m(P) f = \sum\nolimits_{\lambda \in \Lambda_P} m(\lambda) f_\lambda
\end{equation}
We call $m(P)$ a \emph{spectral multiplier}. For any $m: [0,\infty) \to \CC$ and $R > 0$, we define $m_R(\lambda) = m(\lambda / R)$ and consider the family of operators $m_R(P)$. We are able to prove a version of Theorem \ref{zonalconvolutionapplication} in this more general context, provided that the eigenvalues of the operator $P$ lie in an arithmetic progression and the principal symbol $p$ satisfies an appropriate curvature condition.


\begin{theorem} \label{CpVersionOfTheorem}
    Let $M$ be a compact manifold of dimension $d$, and let $P$ be a classical elliptic self-adjoint pseudodifferential operator of order one on $M$ with principal symbol $p: T^* M \to [0,\infty)$. Suppose that
    \begin{enumerate}
    \setlength\itemsep{0.5em}

        \item For each $x \in M$, the cosphere
        \begin{equation}
            S_x^* = \big\{ \xi \in T_x^*M : p(x,\xi) = 1 \big\}
        \end{equation}
        is a hypersurface in $T_x^* M$ with non-vanishing Gaussian curvature.

        \item The eigenvalues of $P$ are contained in an arithmetic progression.
        %

        %


        %
        %
    \end{enumerate}
    Then if $m$ is a bounded, regulated function compactly supported in $(0,\infty)$, then
    \begin{equation} \label{aiowjdowadjowai3424234}
        \sup\nolimits_R \| m_R(P) \|_{L^p(M) \to L^p(M)} \sim C_p(m),
    \end{equation}
    with implicit constants depending on the support of $m$.
\end{theorem}

Theorem \ref{CpVersionOfTheorem} immediately implies Theorem \ref{zonalconvolutionapplication}.

\begin{proof}[Proof of Theorem \ref{zonalconvolutionapplication}]
    Let $\mathcal{H}_k(S^d)$ denote the space of all spherical harmonics of degree $k$. If $\Delta$ is the Laplace-Beltrami operator on $S^d$, then
    \begin{equation}
        \Delta f = -k(k+d-1) f \quad\quad\text{for any $f \in \mathcal{H}_k(S^d)$}.
    \end{equation}
    If we let $a = (d-1)/2$ and $P = (a^2 - \Delta)^{1/2} - a$, then
    \begin{equation}
        Pf = k f \quad\quad\text{for $f \in \mathcal{H}_k(S^d)$}.
    \end{equation}
    Thus $T_R = m_R(P)$. The operator $P$ is a pseudodifferential operator of order $1$ with principal symbol $p(x,\xi) = |\xi|_x$, where $|\xi|_x$ is the norm induced on $T^*M$ by the Riemannian metric on $S^d$. Thus $P$ satisfies the assumptions of Theorem \ref{CpVersionOfTheorem}, since the cospheres of $p$ are all ellipses, and $\Lambda_P = \NN$. But then Theorem \ref{zonalconvolutionapplication} immediately follows from the conclusion of Theorem \ref{CpVersionOfTheorem}. 
\end{proof}


\subsection{Applications to Eigenfunction Expansions}

In light of the uniform boundedness principle, Theorem \ref{CpVersionOfTheorem} also implies results for expansions in eigenvalues. If we assume \eqref{FutureEquation}, then we could actually obtain a complete characterization of the multipliers that induce convergent expansions in eigenvalues, namely, that for any regulated function $m: [0,\infty) \to \CC$, the condition
\begin{equation} \label{partialsumbound}
    \lim\nolimits_{R \to \infty} \| m_R(P) f - f \|_{L^p(M)} = 0 \quad\text{for any $f \in L^p(M)$}
\end{equation}
holds if and only if $m(0) = 1$ and $\sup_{j \in \ZZ} C_p(m_j) < \infty$. As it stands, Theorem \ref{CpVersionOfTheorem} does not imply a result directly comparable to \eqref{partialsumbound}, since if $m$ is compactly supported in $(0,\infty)$, then it is never true that $m(0) = 1$. But Theorem \ref{CpVersionOfTheorem} does imply a weaker result.

\begin{corollary} \label{eigenfunctionexpansioncharacterization}
    Suppose $M$ is a compact manifold of dimension $d$, and $P$ is a classical elliptic self-adjoint pseudodifferential operator of order one on $M$ satisfying the assumptions of Theorem \ref{CpVersionOfTheorem}. If $1 < p < 2 (d-1)/(d+1)$, then for any regulated function $m$ compactly supported in $(0,\infty)$, the condition
    \begin{equation} \label{eigenfunctionexpansionequation}
        \lim\nolimits_{R \to \infty} \| m_R(P) f \|_{L^p(M)} = 0 \quad\text{for any $f \in L^p(M)$}.
    \end{equation}
    holds if and only if $C_p(m) < \infty$.
\end{corollary}
\begin{proof}
    The uniform boundedness principle implies that if \eqref{eigenfunctionexpansionequation} holds, then the operators $m_R(P)$ are uniformly bounded on $L^p(M)$, and Theorem \ref{CpVersionOfTheorem} implies this holds if and only if $C_p(m) < \infty$. Conversely, if $C_p(m) < \infty$, then by Theorem \ref{CpVersionOfTheorem}, the operators $m_R(P)$ are uniformly bounded on $L^p(M)$, and one can obtain \eqref{eigenfunctionexpansionequation} by first proving the result for functions $f$ that are finite linear combinations of eigenfunctions, and then applying an approximation argument since finite linear combinations of eigenfunctions are dense in $L^p(M)$.
\end{proof}

\subsection{A Resulting Transference Theorem} \label{transference}

One may rephrase Theorem \ref{CpVersionOfTheorem} as a \emph{transference theorem}, which gives an equivalence between $L^p$ operator bounds for zonal convolution operators, and $L^p$ operator bounds for Fourier multiplier operators on Euclidean space.

\begin{corollary} \label{maintheoremprime}
    Suppose $M$ is a compact manifold of dimension $d$, and $P$ is a classical elliptic self-adjoint pseudodifferential operator of order one on $M$ satisfying the assumptions of Theorem \ref{CpVersionOfTheorem}. For any regulated function $m$ compactly supported in $(0,\infty)$, consider the radial Fourier multiplier operator $m( |D| )$ on $\RR^d$ defined by
    \begin{equation} \label{FourierMultiplierEquation}
    m \big(|D| \big) f(x) = \int_{\RR^d} m(|\xi|) \widehat{f}(\xi) e^{2 \pi i \xi \cdot x}\; d\xi.
    \end{equation}
    Then for $1 < p < 2(d-1)/(d+1)$,
    %
    %
    %
    \begin{equation}
        \sup\nolimits_R \| m_R(P) \|_{L^p(M) \to L^p(M)} \sim \| m \big(|D| \big) \|_{L^p(\RR^d) \to L^p(\RR^d)}.
    \end{equation}
    %
\end{corollary}
\begin{proof}
    The main result of \cite{HeoandNazarovandSeeger} shows that, for $1 < p < 2 (d-1)/(d+1)$, and for any function $m: [0,\infty) \to \CC$, the radial Fourier multiplier operator defined in \eqref{FourierMultiplierEquation} satisfies
    \begin{equation} \label{HNSResult}
        \| m \big(|D| \big) \|_{L^p(\RR^d) \to L^p(\RR^d)} \sim C_p(m).
    \end{equation}
    Combining this result with \eqref{aiowjdowadjowai3424234} of Theorem \ref{CpVersionOfTheorem} immediately yields the claim.
\end{proof}

Transference principles that control $L^p$ bounds for Fourier multipliers on $\RR^d$ in terms of $L^p$ bounds for spectral multipliers on a manifold are classical. An essentially complete result was proved in a 1974 result of Mitjagin \cite{Mitjagin}, which states that for all $p \in [1,\infty]$, for any compact manifold $M$, and for any regulated function $m: [0,\infty) \to \CC$,
\begin{equation} \label{mitjagininequalityforlaplacian}
    \sup\nolimits_R \big\| m_R \big( P \big) \big\|_{L^p(M) \to L^p(M)} \gtrsim \big\| m \big( |D| \big) \big\|_{L^p(\RR^d) \to L^p(\RR^d)}.
\end{equation}
On the other hand, Corollary \ref{maintheoremprime} is the first result in the literature to obtain a transference principle in the opposite direction, i.e. controlling $L^p$ operator bounds for spectral multipliers on a manifold in terms of $L^p$ operator bounds for Fourier multipliers on $\RR^d$. No such principle has been proved for any $p \neq 2$ and any elliptic, self-adjoint elliptic pseudo-differential operator $P$ on a compact manifold.\footnote[1]{In making this statement we excuse the relatively simple case where $M = \TT^d$ is a flat torus and $P = \sqrt{-\Delta}$. Spectral multipliers for this $P$ are Fourier multipliers for the Fourier series on $\TT^d$, and so Poisson summation methods can be applied to yield a relatively simple proof of a transference bound for all $p \in [1,\infty)$. See Theorem 3.6.7 of \cite{Grafakos} for more details.}

Obtaining transference principles controlling $L^p$ bounds for Fourier multipliers in terms of $L^p$ bounds on spectral multipliers is heuristically simpler to establish than the reverse. This is best shown in the model case where $M$ is a Riemannian manifold and the principal symbol of $P$ is equal to $|\xi|_x$, the induced norm on $T^* M$ from the Riemannian metric on $M$. We can write $m_R(P) = m(P_R)$, where $P_R = P / R$ is a pseudodifferential operator whose principal symbol is the norm on $T^* M$ associated with the \emph{dilated metric} $g_R = R^2 g$. As $R \to \infty$, the metric $g_R$ gives the manifold $M$ less curvature and more volume, and so we might expect $M$ equipped with $P$ to behave more and more like $\RR^d$ equipped with the analogous operator $|D| = \sqrt{-\Delta}$. Indeed, Mitjagin's proof, very roughly speaking, consists of showing that, in coordinates, and with a limit taken in an appropriate sense,
\begin{equation} \label{limitingmitjaginequation}
    \lim\nolimits_{R \to \infty} m_R(P) = m \big(|D| \big).
\end{equation}
Taking operator norms on both sides of \eqref{limitingmitjaginequation} yields inequality \eqref{mitjagininequalityforlaplacian}.

On the other hand, it seems more difficult to obtain the converse inequality
\begin{equation} \label{upperboundmainresult}
    \sup\nolimits_R \big\| m_R(P) \big\|_{L^p(M) \to L^p(M)} \lesssim \big\| m \big(|D| \big) \big\|_{L^p(\RR^d) \to L^p(\RR^d)}
\end{equation}
This should be expected, since unlike \eqref{mitjagininequalityforlaplacian}, it is not immediately geometrically intuitive why an $L^p$ bound for a Fourier multiplier on a flat space should imply uniform $L^p$ bounds for a family of spectral multipliers on curved spaces.  Moreover, recent results on the failure of the Kakeya conjecture in three dimensional manifolds with non-constant sectional curvature \cite{DaiGongGuoZhang} provide weak evidence that the range of $L^p$ spaces to which inequality \eqref{upperboundmainresult} can be obtained may be more sensitive to the geometry of a manifold $M$ than the range of $L^p$ spaces under which \eqref{mitjagininequalityforlaplacian} holds. Obtaining bounds of the form \eqref{upperboundmainresult} therefore requires a more subtle analysis than required to obtain bounds of the form \eqref{mitjagininequalityforlaplacian}.

\subsection{Applications To Other Pseudodifferential Operators}

The assumptions of Theorem \ref{CpVersionOfTheorem} are relatively restrictive, but there are several other settings than the context of Theorem \ref{zonalconvolutionapplication} where Theorem \ref{CpVersionOfTheorem} implies new results. Recalling the proof of Theorem \ref{zonalconvolutionapplication} from Theorem \ref{CpVersionOfTheorem}, we analyzed the operator
\begin{equation} \label{periodicshiftequation}
    P = (a^2 - \Delta)^{1/2} - a
\end{equation}
on $S^d$, where $a = (d-1)/2$ and $\Delta$ is the Laplace-Beltrami operator on $S^d$. A similar construction works on the analysis of Laplace-Beltrami operators on projective spaces over any normed division algebra. Namely, if $M$ is such a projective space, then the eigenvalues of the Laplace-Beltrami operator on $M$ are quadratic functions with integer coefficients, and so we can find $a \in \QQ$ such that if we interpret \eqref{periodicshiftequation} as defining an operator on $M$, then $\Lambda_P = \NN$. Theorem \ref{CpVersionOfTheorem} then applies to $P$. The table below gives the choice of $a$ for each of these projective spaces.

\vspace{0.3em}
\begin{center}
\begin{tabular}{|c|c|c|}
    \hline
    Space & Eigenvalues of $\Delta$ & Choice of $a$\\
    \hline
    $\RR \PP^d$ & $-2k(2k+d-1)$ & $(d-1)/2$ \\
    $\CC \PP^d$ & $-4k(k+d)$ & $d$ \\
    $\HH \PP^d$ & $-4k(k+2d+1)$ & $2d+1$ \\
    $\mathbb{O} \PP^2$ & $-4k(k+11)$ & $11$ \\
    \hline
\end{tabular}
\end{center}
\vspace{0.3em} The spaces $\mathbb{S}^d$, $\RR \PP^d$, $\CC \PP^d$, $\HH \PP^d$, and $\mathbb{O} \PP^2$ give all of the \emph{compact rank one symmetric spaces} (see Chapter 3 of \cite{Besse} for a definition and classification). Thus Theorem \ref{CpVersionOfTheorem} can be applied to obtain bounds for explicit perturbations of the Laplace-Beltrami operator on all such spaces. One can think of the multiplier operators for such a perturbation as multiplier operators for expansions of functions for analogues of spherical harmonics on the other rank one symmetric spaces.

In the examples above, the spectral multipliers for the operators $P$ above are really just reparameterizations of spectral multipliers of $\sqrt{-\Delta}$, and so Theorem \ref{CpVersionOfTheorem} gives new bounds for multipliers of the Laplacian. More precisely, if $P = \sqrt{a^2 - \Delta} - a$ for some constant $a > 0$, then $\sqrt{-\Delta} = \phi(P)$, where $\phi(\lambda) = \sqrt{\lambda^2 + 2 a \lambda}$. Thus if $m: [0,\infty) \to \CC$ is a regulated function compactly supported in $(0,\infty)$, then Theorem \ref{CpVersionOfTheorem} implies
\begin{equation} \label{LaplacianResult}
    \| m( \sqrt{-\Delta} ) \|_{L^p(M) \to L^p(M)} = \| (m \circ \phi)(P) \|_{L^p(M) \to L^p(M)} \lesssim \sup\nolimits_j C_p( (m \circ \phi)_j ).
\end{equation}
In this case simple estimates can be used to obtain $L^p$ boundedness using the compact support of $m$. Namely, if $m$ is supported on $(A,B)$, then the triangle inequality, Weyl Law, and $L^p$ estimates for eigenfunctions tell us that
%
\begin{equation} \label{trivialbound}
    \| m( \sqrt{-\Delta} ) \|_{L^p(M) \to L^p(M)} \lesssim B^{\delta(p)} \| m \|_{L^\infty[0,\infty)},
\end{equation}
for an exponent $\delta(p) > 0$ (see Lemma \ref{lowjLemma} for more details along these lines). Thus $C_p(m)$ need not be finite in order for $m( \sqrt{-\Delta} )$ to be bounded. However, if $B$ is large \eqref{LaplacianResult} likely gives more efficient upper bounds. Using Remark \ref{ABRemark}, we obtain
\begin{equation} \label{aodijaowidjwaio}
    \| m( \sqrt{-\Delta} ) \|_{L^p(M) \to L^p(M)} \lesssim |\log(B/A)|\; \sup\nolimits_j C_p(m_j),
\end{equation}
where $m_j(\lambda) = \chi(\lambda) m(\phi(2^j \lambda))$, which is a much better bound than \eqref{trivialbound} if $B$ is large relative to $\sup\nolimits_j C_p(m_j)$. Moreover, in a follow up paper, we hope to obtain \eqref{FutureEquation}, and thus a bound independent of $A$ and $B$ if $\sup\nolimits_j C_p(m_j) < \infty$. 

On the other hand, Theorem \ref{CpVersionOfTheorem} \emph{cannot} characterize those functions $m$ such that the spectral multipliers $m_R( \sqrt{-\Delta} )$ are uniformly bounded on $L^p(M)$, and this problem remains open on all compact Riemannian manifolds $M \neq \TT^d$ for $p \neq 2$. Theorem \ref{CpVersionOfTheorem} allows us to characterize when families of operators of the form $m( \phi(P/R) )$ are bounded, but the operators $m_R \big( \sqrt{-\Delta} \big) = m( \phi(P) / R )$ are \emph{not} of this form. If $m_{R,\phi}(\lambda) = m( \phi( R \lambda) / R)$ then we can use Theorem \ref{CpVersionOfTheorem} to conclude that
\begin{equation}
\begin{split}
    \sup\nolimits_R \| m_R \big( \sqrt{-\Delta} \big) \|_{L^p(M) \to L^p(M)} &= \sup\nolimits_R \| m_{R,\phi}(P) \|_{L^p(M) \to L^p(M)}\\
    &\lesssim \sup\nolimits_R C_p(m_{R,\phi}).
\end{split}
\end{equation}
However, $\sup\nolimits_R C_p(m_{R,\phi})$ is in general incomparable to $C_p(m)$, and it is not immediately apparent whether it is necessary for $\sup\nolimits_R C_p(m_{R,\phi})$ to be finite in order for the operators $m_R(\sqrt{-\Delta})$ to be uniformly bounded, or whether having $C_p(m)$ finite suffices, as has been verified in the special case of zonal input functions \cite{Alladi}.

There is one other class of operators to Theorem \ref{CpVersionOfTheorem} applies, obtained by considering Hamiltonian flows associated with pseudodifferential operators. Let $P_0$ be any classical elliptic self-adjoint pseudodifferential operator of order one on a compact manifold $M$, with principal symbol $p: T^* M \to [0,\infty)$. Then $p$ generates a Hamiltonian flow $\{ \alpha_t \}$ on $T^* M$ given in coordinates by the equations
\begin{equation} \label{qoiJIOJdoiwajfoiawjfoi}
    x'(t) = (\partial_\xi p)(x,\xi) \quad\text{and}\quad \xi'(t) = - (\partial_x p)(x,\xi).
\end{equation}
If $\Lambda_{P_0}$ is contained in an arithmetic progression $\{ \lambda_0 + k / \Pi : k \geq 0 \}$ for some $\Pi > 0$, then this flow must be periodic, of period $\Pi$ (this can be seen by analyzing the wavefront sets of the propagators $e^{2 \pi i t P_0}$ via the Propogation of Singularities Theorem). 
The converse of this argument is also true, modulo a perturbation. Namely, if $P_0$ is an operator with constant sub-principal symbol, and it's corresponding Hamiltonian flow is periodic with period $\Pi$, then Lemma 29.2.1 of \cite{Hormander4} implies that there is $\lambda \in \RR$, and an operator $P$ commuting with $P_0$, such that the eigenvalues of $P$ are of the form $\{ k \Pi^{-1} : k \geq 0 \}$, and $P - (P_0 + \lambda I )$ is a pseudodifferential operator of order $-1$. If in addition, the cospheres associated with the principal symbol $p$ are ellipses, then $P$ satisfies all assumptions of Theorem \ref{CpVersionOfTheorem}, and so Theorem \ref{CpVersionOfTheorem} implies $L^p$ bounds for multipliers of $P$.

The primary examples of operators to which the above method applies are of the form $P_0 = \sqrt{-\Delta}$, where $\Delta$ is the Laplace-Beltrami operator on a compact Riemannian manifold $M$ with periodic geodesic flow.\footnote[1]{In a certain generalized sense, pertubations of Laplace-Beltrami operators are the \emph{only} operators to which we can apply Theorem \ref{CpVersionOfTheorem}. If $p: T^* M \to [0,\infty)$ is the principal symbol of an operator $P$ and has a cosphere with non-vanishing Gaussian curvature, then $p$ induces a Finsler metric on $M$, such that the Hamiltonian flow induced by $p$ is the geodesic flow for the Finsler metric (see Section \ref{estimatesforwavepackets} for more details). There is no canonical Laplace-Beltrami operator on a Finsler manifold, but the operator $P$ might be thought of as an analogue of $\sqrt{-\Delta}$ on $M$, at least up to a lower order pertubation.} The subprincipal symbol of $P_0$ is then always zero, so one may apply the results of the previous paragraph to perturb $P_0$ to obtain an operator $P$ to which Theorem \ref{CpVersionOfTheorem} applies. The only currently known examples of Riemannian manifolds with periodic geodesic flow other than the compact rank one symmetric spaces, which we have already encountered, are the \emph{Zoll manifolds for the sphere}, an infinite dimensional family of manifolds which are diffeomorphic to the sphere, but with a perturbed metric so that the manifolds are not \emph{isometric} to the sphere (see Chapter 4 of \cite{Besse} for more details). Zoll manifolds provide examples of manifolds $M$ and operators $P$ to which Theorem \ref{CpVersionOfTheorem} applies, but for which $M$ lacks a large symmetry group under which the operator $P$ is invariant. For a generic Zoll manifold there need not be any such symmetries.

One difference with this more general method than with the compact rank one symmetric spaces is that the class of spectral multipliers of $P_0$ is not necessarily equal to the class of spectral multipliers of $P$, and so Theorem \ref{CpVersionOfTheorem} cannot be applied to analyze all multipliers of $P_0$. Indeed, our assumptions do imply that the eigenvalues of $P_0$ occur in clusters contained in intervals of the form $I_k = [k \Pi^{-1} - C / k, k \Pi^{-1} + C/k]$ for some $C > 0$, but there can be multiple eigenvalues in this interval, and the eigenspaces of these eigenvalues will be identified when studying the operator $P$. One notable exception where such a situation does not occur is if each interval $I_k$ only contains a single eigenvalue (necessarily of high multiplicity) for suitably large $k$. Such manifolds are called \emph{maximally degenerate Zoll manifolds} \cite{Zelditch}, and then an analogue of \eqref{LaplacianResult} applies.

\subsection{Related Results} Boundedness of multipliers for spherical harmonic expansions have been studied throughout the 20th century. Classical methods involve the analysis of special functions and orthogonal polynomials, culminating in Bonami and Clerc's extension of the Marcinkiewicz multiplier theorem to multipliers for spherical harmonic expansions \cite{BonamiClerc}. Methods changed in the 1960s, when H\"{o}rmander introduced the powerful method of Fourier integral operators, which he used to obtain $L^p$ bounds for Bochner-Riesz multiplier expansions for spherical harmonics \cite{HormanderRiesz}, with later developments on $S^d$ and on other compact manifolds by Sogge \cite{SoggeSphericalHarmonics,SoggeSpectralClusters,SoggeRieszMeans}, Christ and Sogge \cite{ChristandSogge}, Seeger and Sogge \cite{SeegerSoggeBochnerRiesz}, Seeger \cite{SeegerEndpointMultipliers}, Tao \cite{Tao}, and Kim \cite{KimSpectral}. None of these results were able to \emph{completely} characterize the uniform $L^p$ boundedness of the multipliers $\{ T_R \}$, though the result of \cite{KimSpectral} is closest, characterizing boundedness at the level of Besov spaces, which we discuss in more detail. Our methods are heavily inspired by recent results on characterization of $L^p$ boundedness for radial Fourier multiplier operators on $\RR^d$ \cite{GarrigosandSeeger,HeoandNazarovandSeeger,Cladek,KimQuasiradial}, especially \cite{HeoandNazarovandSeeger}, which might be expected since the transference principle of \cite{Mitjagin} shows the results of \cite{HeoandNazarovandSeeger} are implied by the results of this paper.

In the remainder of this section, we compare and contrast Theorem \ref{CpVersionOfTheorem} with the result of Kim \cite{KimSpectral}. For a classical elliptic self-adjoint pseudo-differential operator $P$ of order one on a manifold $M$ whose principal symbol has cospheres with nonvanishing Gaussian curvature, but not necessarily satisfying the periodic spectrum assumption of Theorem \ref{CpVersionOfTheorem}, \cite{KimSpectral} shows that in the larger range $1 < p < 2(d+1)/(d+3)$, and for any regulated function $m$ with compact support in $(0,\infty)$,
\begin{equation} \label{KimResult123}
    \sup\nolimits_{R > 0} \| m_R(P) \|_{L^p(M) \to L^p(M)} \lesssim \| m \|_{L^p[0,\infty)} + D_p(m),
\end{equation}
where
\begin{equation}
    D_p(m) = \bigg( \sum\nolimits_{j \in \ZZ} \bigg( \int_{2^j}^{2^{j+1}} \Big[ \langle t \rangle^{\alpha(p) + |1/p - 1/2|} |\widehat{m}(t)| \Big]^2\; dt \bigg)^{p/2}\; dt \bigg)^{1/p}.
\end{equation}
%
%
%
%
%
%
%
%
%
%
One has
\begin{equation} \label{AWOIDJWAOIFJWAOI34109412094319023}
    C_p(m) \lesssim \| m \|_{L^p[0,\infty)} + D_p(m)
\end{equation}
so that Theorem \ref{CpVersionOfTheorem} gives a stronger bound than than \eqref{KimResult123} in situations where both apply. We can obtain \eqref{AWOIDJWAOIFJWAOI34109412094319023} by applying H\"{o}lder's inequality for each $j$. However, a more intuitive, but less rigorous explanation follows from noticing that $D_p(m) < \infty$, roughly speaking, guarantees that $m$ has $\alpha(p) + |1/p - 1/2|$ derivatives in $L^2$, whereas $C_p(m) < \infty$, even more roughly speaking, says that $m$ has $\alpha(p)$ derivatives in $L^{p'}$. So one might expect a bound of the form \eqref{AWOIDJWAOIFJWAOI34109412094319023} by Sobolev embedding heuristics.

%
%
%

Kim's result follows some similar strategies to our paper. But the stronger $L^2$ estimates on the multiplier allow one to reduce the large time behavior of the wave equation to $L^2$ orthogonality and discrete analogues of the Tomas-Stein restriction theorem originally proved in \cite{SoggeSpectralClusters}, which for $1 \leq p < 2(d-1)/(d+1)$, are of the form
\begin{equation}
    \big\| \mathbb{I}_{[k,k+1]}(P) \big\|_{L^2(M) \to L^p(M)} \lesssim k^{\alpha(p) - 1/p'}.
\end{equation}
%
%
%
%
%
%
%
%
This method generalizes a technique previously applied to the study of Bochner-Riesz multipliers \cite{SoggeRieszMeans}. We are unable to apply this method in the proof of Theorem \ref{CpVersionOfTheorem}; reducing $L^p$ norms to $L^2$ norms in order to apply the Stein-Tomas theorem is precisely what causes a loss of $|1/p - 1/2|$ derivatives required for $D_p(m)$ to be finite, when compared to the finiteness of $C_p(m)$. One can view the calculations in Section \ref{regime2finalsection} of this paper, which reduce large time estimates of the half-wave equation to local smoothing bounds, as a substitute for this reduction, which avoids the direct use of a reduction to $L^2$ estimates. But these bounds require a more robust argument, which requires a greater understanding of solutions to the half-wave equation for large times, and thus we are forced to make the stricter periodic spectrum assumptions.

\subsection{Summary of Proof} To conclude the introductory sections of the paper, we summarize the various methods which occur in the proof of Theorem \ref{CpVersionOfTheorem}. In Lemma \ref{lowerboundlemma} of Section \ref{PrelimSetup}, we prove that
\begin{equation} \label{nonoppositeinequality}
    \sup\nolimits_R \| m_R(P) \|_{L^p(M) \to L^p(M)} \gtrsim C_p(m).
\end{equation}
We obtain \eqref{nonoppositeinequality} via a relatively simple argument, as might be expected from the discussion in Section \ref{transference}. The main difficulty is obtaining the opposite inequality
\begin{equation} \label{oppositeinequality}
    \sup\nolimits_R \| m_R(P) \|_{L^p(M) \to L^p(M)} \lesssim C_p(m).
\end{equation}
Unlike in the study of Euclidean multipliers, the behavior of the multipliers $m_R(P)$ on a \emph{compact} manifold for $R \lesssim 1$ is relatively benign. In Lemma \ref{lowjLemma} of Section \ref{PrelimSetup} that same section, another simple argument shows
\begin{equation} \label{TrivialLowFrequencyBound}
    \sup\nolimits_{R \leq 1} \left\| m_R(P) \right\|_{L^p(M) \to L^p(M)} \lesssim C_p(m).
\end{equation}
In that section we also setup notation for the rest of the paper. As a consequence of Proposition \ref{TjbLemma}, stated at the end of Section \ref{PrelimSetup},
we can conclude that
\begin{equation} \label{dyadicMainReulst}
    \sup\nolimits_{R \geq 1} \| m_R(P) \|_{L^p(M) \to L^p(M)} \lesssim C_p(m),
\end{equation}
The upper bound \eqref{dyadicMainReulst} requires a more in depth analysis than \eqref{TrivialLowFrequencyBound}. Using the fact that $m$ is regulated, 
we apply the Fourier inversion formula to write
\begin{equation}
    m_R(P) = \int_{-\infty}^\infty R \widehat{m}(Rt) e^{2 \pi i t P}\; dt,
\end{equation}
where $e^{2 \pi i t P}$ are wave propogators, which as $t$ varies, give 
%
%
solutions to the half-wave equation $\partial_t = 2 \pi i P$ on $M$. Studying $m_R(P)$ thus reduces to studying certain 'weighted averages' of the propagators $\{ e^{2 \pi i t P} \}$. We prove Proposition \ref{TjbLemma} using several new estimates for understanding these averages, including:
\begin{itemize}
    \item[(A)] Quasi-orthogonality estimates for averages of solutions to the half-wave equation on $M$, discussed in Section \ref{estimatesforwavepackets}, which arise from a connection between the theory of pseudodifferential operators whose principal symbols satisfy the curvature condition of Theorem \ref{CpVersionOfTheorem}, and Finsler metrics on the manifold $M$.

    \item[(B)] Variants of the density-decomposition arguments first used in \cite{HeoandNazarovandSeeger}, described in Section \ref{regime1firstsection}, which apply the quasi-orthogonality estimates obtained in Section \ref{estimatesforwavepackets} with a geometric argument which controls the `small time behavior' of solutions to the half-wave equation.

    \item[(C)] A new strategy to reduce the `large time behavior' of the half-wave equation to an endpoint local smoothing inequality for the half-wave equation on $M$, described in Section \ref{regime2finalsection}.
\end{itemize}
Equations \eqref{TrivialLowFrequencyBound} and \eqref{dyadicMainReulst} immediately imply \eqref{oppositeinequality}, which together with \eqref{nonoppositeinequality} completes the proof of Theorem \ref{CpVersionOfTheorem}.

\section{Preliminary Setup} \label{PrelimSetup}

%
%
%
%
%
%
%
%
%
%
%
%
%
%

We now begin with the details of the proof of Theorem \ref{CpVersionOfTheorem}, following the path laid out at the end of the introduction. To begin with, we prove the lower bound \eqref{nonoppositeinequality}.

\begin{lemma} \label{lowerboundlemma}
    Suppose $M$ is a compact manifold of dimension $d$, and $P$ is a classical elliptic self-adjoint pseudodifferential operator of order one on $M$ satisfying the assumptions of Theorem \ref{CpVersionOfTheorem}. Then for $1 < p < 2(d-1)/(d+1)$, and for any regulated function $m: [0,\infty) \to \CC$,
    \begin{equation}
        \sup\nolimits_R \| m_R(P) \|_{L^p(M) \to L^p(M)} \gtrsim C_p(m).
    \end{equation}
\end{lemma}
\begin{proof}
Define the quasiradial Fourier multiplier $m( p(x_0,D) )$ on $\RR^d$ by
\begin{equation}
    m( p(x_0,D) ) f(x) = \int_{\RR^d} m(p(x_0,\xi)) \widehat{f}(\xi) e^{2 \pi i \xi \cdot x}\; d\xi.
\end{equation}
A result of Mitjagin \cite{Mitjagin} 
states that for any regulated $m$, and any $x_0 \in M$,
\begin{equation}
\label{mitjagininequality}
    \sup\nolimits_R \big\| m_R(P) \big\|_{L^p(M) \to L^p(M)} \gtrsim \big\| m \big( p(x_0,D) \big) \big\|_{L^p(\RR^d) \to L^p(\RR^d)}.
\end{equation}
Theorem 1.1 of \cite{KimQuasiradial} implies that
\begin{equation} \label{kiminequality}
    \big\| m \big( p(x_0,D) \big) \big\|_{L^p(\RR^d) \to L^p(\RR^d)} \sim C_p(m) \quad\text{and}\quad \big\| m\big( |D| \big) \big\|_{L^p(\RR^d) \to L^p(\RR^d)} \sim C_p(m).
\end{equation}
Putting together \eqref{mitjagininequality} and \eqref{kiminequality} then proves the claim. 
\end{proof}

In the remainder of the article, we focus on proving the upper bound \eqref{oppositeinequality}. For this purpose, we may assume without loss of generality that $\text{supp}(m) \subset [1,2]$. Let us briefly justify why. 
 Suppose we have proved
\begin{equation} \label{specialcasedyadicCp}
    \sup\nolimits_R \| m_R(P) \|_{L^p \to L^p} \sim C_p(m) \quad\text{if $\text{supp}(m) \subset [1,2]$.}
\end{equation}
Let $m$ be an arbitrary function whose support is a compact subset of $(0,\infty)$. Fix $\chi \in C_c^\infty(0,\infty)$ with $\text{supp}(\chi) \subset [1,2]$ and with $\sum \chi( \cdot / 2^j ) = 1$. If we set $m_j(\lambda) = \chi(\lambda) m(2^j \lambda)$, then $\text{supp}(m_j) \subset [1,2]$ for all $j$, and $m(\lambda) = \sum m_j(\lambda/2^j)$. Then
\begin{equation}
    \widehat{m}_j(t) = \widehat{\chi} * 2^{-j} \widehat{m}( \cdot / 2^j ),
\end{equation}
%
%
and the rapid decay of $\widehat{\chi}$ implies that $C_p(m_j) \lesssim_j C_p(m)$. Thus \eqref{specialcasedyadicCp} implies
\begin{equation}
\begin{split}
    &\sup\nolimits_R \| m_j(P/2^j R) \|_{L^p(M) \to L^p(M)}\\
    &\quad\quad= \sup\nolimits_R \| m_j(P/R) \|_{L^p(M) \to L^p(M)} \lesssim_j C_p(m_j) \lesssim C_p(m).
\end{split}
\end{equation}
Now the triangle inequality, summing over finitely many $j$, implies
\begin{equation}
    \sup\nolimits_R \| m(P/R) \|_{L^p(M) \to L^p(M)} \lesssim \sum\nolimits_j \sup\nolimits_R \| m_j(P/2^j R) \|_{L^p(M) \to L^p(M)} \lesssim C_p(m),
\end{equation}
which proves Theorem \ref{CpVersionOfTheorem} in general.

Given a multiplier $m$ supported on $[1/2,2]$, we define $T_R = m(P/R)$. We fix some geometric constant $\varepsilon_M \in (0,1)$, matching the constant given in the statement of Theorem \ref{TjbLemma}. Our goal is then to prove inequalities \eqref{TrivialLowFrequencyBound} and \eqref{dyadicMainReulst}. Proving \eqref{TrivialLowFrequencyBound} is simple because the operators $T_R$ are smoothing operators, uniformly for $0 < R < 1$, and $M$ is a compact manifold.

\begin{lemma} \label{lowjLemma}
    Let $M$ be a compact $d$-dimensional manifold, let $P$ be a classical elliptic self-adjoint pseudodifferential operator of order one. If $1 < p < 2d/(d+1)$, and if $m: [0,\infty) \to \CC$ is a regulated function with $\text{supp}(m) \subset [1/2,2]$, then
    \begin{equation}
        \sup\nolimits_{R \leq 1} \| T_R \|_{L^p(M) \to L^p(M)} \lesssim C_p(m).
    \end{equation}
\end{lemma}
\begin{proof}
    Let $T_R = m_R(P)$. The set $\Lambda_P \cap [0,2]$ is finite. For each $\lambda \in \Lambda_P$, choose a finite orthonormal basis $\mathcal{E}_\lambda$. Then we can write
    \begin{equation}
        T_R = \sum\nolimits_{\lambda \in \Lambda_P} \sum\nolimits_{e \in \mathcal{E}_\lambda} \langle f, e \rangle e.
    \end{equation}
    Since $\mathcal{V}_\lambda \subset C^\infty(M)$, H\"{o}lder's inequality implies
    \begin{equation}
    \begin{split}
        \| \langle f, e \rangle e \|_{L^p(M)} &\leq \| f \|_{L^p(M)} \| e \|_{L^{p'}(M)} \| e \|_{L^p(M)} \lesssim_\lambda \| f \|_{L^p(M)}.
    \end{split}
    \end{equation}
    But this means that
    \begin{equation}
        \left\| T_R f \right\|_{L^p(M)} \leq \sum\nolimits_{\lambda \in \Lambda_P \cap [0,2]} \sum\nolimits_{e \in \mathcal{E}_\lambda} |m(\lambda/R)| \| \langle f, e \rangle e \|_{L^p(M)} \lesssim \| m \|_{L^\infty[0,\infty)}.
    \end{equation}
    The proof is completed by noting that for $1 < p < 2d/(d+1)$, the Sobolev embedding theorem guarantees that $\| m \|_{L^\infty[0,\infty)} \lesssim C_p(m)$.
\end{proof}

Lemma \ref{lowjLemma} is relatively simple to prove because the operators $\{ T_R: 0 \leq R \leq 1 \}$ are supported on a common, finite dimensional subspace of $L^2(M)$. We thus did not have to perform any analysis of the interactions between different eigenfunctions, because the triangle inequality is efficient enough to obtain a finite bound. For $R > 1$, things are not so simple, and so proving \eqref{dyadicMainReulst} requires a more refined analysis of multipliers than \eqref{TrivialLowFrequencyBound}. A standard method, originally due to H\"{o}rmander \cite{Hormander2}, is to apply the Fourier inversion formula to write, for a regulated function $h: \RR \to \CC$,
\begin{equation}
    h(P) = \int_{-\infty}^\infty \widehat{h}(t) e^{2 \pi i t P}\; dt,
\end{equation}
where $e^{2 \pi i t P}$ is the spectral multiplier operator on $M$ which, as $t$ varies, gives solutions to the half-wave equation $\partial_t = 2 \pi i P$ on $M$. Thus we conclude that
\begin{equation}
    T_R = \int_{-\infty}^\infty R \widehat{m}(R t) e^{2 \pi i t P}\; dt.
\end{equation}
Writing $T_R$ in this form, we are  lead to obtain estimates for averages of the wave equation tested against a relatively non-smooth function $R \widehat{m}(Rt)$. 

Fix a bump function $q \in C_c^\infty(\RR)$ with $\supp(q) \subset [1/2,4]$ and $q(\lambda) = 1$ for $\lambda \in [1,2]$, and define $Q_R = q(P/R)$. 
We write
\begin{equation}
    T_R = Q_R \circ T_R \circ Q_R = \int_{\RR} R \widehat{m}(R t) (Q_R \circ e^{2 \pi i t P} \circ Q_R)\; dt,
\end{equation}
and view the operators $(Q_R \circ e^{2 \pi i t P} \circ Q_R)$ as `frequency localized' wave propagators.

Replacing the operator $P$ by $aP + b$ for appropriate $a,b \in \RR$, we may assume without loss of generality that all eigenvalues of $P$ are integers. It follows that $e^{2 \pi i (t + n) P} = e^{2 \pi i t P}$ for any $t \in \RR$ and $n \in \ZZ$. Let $I_0$ denote the interval $[-1/2,1/2]$. We may then write
\begin{equation}
    T_R = \int_{I_0} b_R(t) (Q_R \circ e^{2 \pi i tP} \circ Q_R)\; dt,
\end{equation}
where $b_R: I_0 \to \CC$ is the periodic function
\begin{equation}
    b_R(t) = \sum\nolimits_{n \in \ZZ} R \widehat{m}(R (t + n)).
\end{equation}
We split our analysis of $T_R$ into two regimes: regime $\text{I}$ and regime $\text{II}$. In regime $\text{I}$, we analyze the behaviour of the wave equation over times $0 \leq |t| \leq \varepsilon_M$ by decomposing this time interval into length $1/R$ pieces, and analyzing the interactions of the wave equations between the different intervals. In regime $\text{II}$, we analyze the behaviour of the wave equation over times $\varepsilon_M \leq |t| \leq 1$. Here we need not perform such a decomposition, since the boundedness of $C_p(m)$ gives better control on the function $b_R$ over these times.

\begin{lemma} \label{decompositionLemma}
    Fix $\varepsilon > 0$. Let $\mathcal{T}_R = \ZZ/R \cap [-\varepsilon, \varepsilon]$ and define $I_t = [t - 1/R, t + 1/R]$. For a function $m: [0,\infty) \to \CC$, define a periodic function $b: I_0 \to \CC$ by setting
    \begin{equation}
        b(t) = \sum\nolimits_{n \in \ZZ} R \widehat{m}(R(t + n)).
    \end{equation}
    Then we can write $b = \left( \sum\nolimits_{t_0 \in \mathcal{T}_R} b_{t_0}^I \right) + b^{II}$, where
    \begin{equation} \supp(b_{t_0}^I) \subset I_{t_0} \quad\text{and}\quad \supp(b_R^{II}) \subset I_0 \smallsetminus [-\varepsilon,\varepsilon].
    \end{equation}
    Moreover, we have
    \begin{equation} \label{DKAPDKAWIODJAWOI}
        \left( \sum\nolimits_{t_0 \in \mathcal{T}_R} \Big[ \| b^I_{t_0} \|_{L^p(I_0)} \langle R t_0 \rangle^{\alpha(p)} \Big]^p \right)^{1/p} \lesssim R^{1/p'} C_p(m)
    \end{equation}
        and
    \begin{equation} \label{DWAIOJDAOIWDJWAIODJIOJD}
        \| b^{II} \|_{L^p(I_0)} \lesssim R^{1/p' - \alpha(p)} C_p(m).
    \end{equation}
\end{lemma}

The proof is a simple calculation which we relegate to the appendix.

The following proposition, a kind of $L^p$ square root cancellation bound, implies \eqref{dyadicMainReulst} once we take Lemma \ref{decompositionLemma} into account. 
Since we already proved \eqref{TrivialLowFrequencyBound}, proving this proposition completes the proof of Theorem \ref{CpVersionOfTheorem}.

\begin{prop} \label{TjbLemma}
    Let $P$ be a classical elliptic self-adjoint pseudodifferential operator of order one on a compact manifold $M$ of dimension $d$ satisfying the assumptions of Theorem \ref{CpVersionOfTheorem}. Fix $R > 0$ and suppose $1 < p < 2 (d-1)/(d+1)$. Then there exists $\varepsilon_M > 0$ with the following property. Consider any function $b: I_0 \to \CC$, and suppose we can write $b = \sum\nolimits_{t_0 \in \mathcal{T}_R} b_{t_0}^I + b^{II}$, where $\supp(b_{t_0}^I) \subset I_{t_0}$ and $\supp(b_R^{II}) \subset I_0 \smallsetminus [-\varepsilon_M,\varepsilon_M]$. Define operators $T^I = \sum\nolimits_{t_0 \in \mathcal{T}_R} T^I_{t_0}$ and $T^{II}$, where
    \[ T_{t_0}^I = \int b_{t_0}^I(t) ( Q_R \circ e^{2 \pi i tP} \circ Q_R )\; dt\ \ \text{and}\ \ T^{II} = \int b^{II}(t) ( Q_R \circ e^{2 \pi i tP} \circ Q_R)\; dt. \]
    Then
    \begin{equation} \label{ejqwoifjeoifjwqoifjwqoi}
        \| T^I \|_{L^p \to L^p} \lesssim R^{-1/p'} \left( \sum\nolimits_{t_0 \in \mathcal{T}_R} \Big[ \| b^I_{t_0} \|_{L^p(I_0)} \langle R t_0 \rangle^{\alpha(p)} \Big]^{p} \right)^{1/p}
    \end{equation}
    and
    \begin{equation} \label{DPOIJAOIWDJQWIOFJQOIVJIEOVNFNJNVNV}
        \| T^{II} \|_{L^p \to L^p} \lesssim R^{\alpha(p) - 1/p'} \| b^{II} \|_{L^p(I_0)}.
    \end{equation}
\end{prop}

Proposition \ref{TjbLemma} splits the main bound of the paper into two regimes: regime $\text{I}$ and regime $\text{II}$. Noting that we require weaker bounds in \eqref{DPOIJAOIWDJQWIOFJQOIVJIEOVNFNJNVNV} than in \eqref{ejqwoifjeoifjwqoifjwqoi}, the operator $T^{II}$ will not require as refined an analysis as for the operator $T^I$, and we obtain bounds on $T^{II}$ by a reduction to an endpoint local smoothing inequality in Section \ref{regime2finalsection}. On the other hand, to obtain more refined estimates for the $L^p$ norms of quantities of the form $f = T^I u$, we consider a decompositions of the form $u = \sum_{x_0} u_{x_0}$, where $u_{x_0}: M \to \CC$ is supported on a ball $B(x,1/R)$ of radius $1/R$ centered at $x_0$. We then have $\smash{f = \sum\nolimits_{(x_0,t_0)} f_{x_0,t_0}}$, where $f_{x_0,t_0} = T^I_{t_0} u_{x_0}$. To control $f$ we must establish an $L^p$ square root cancellation bound for the functions $\{ f_{x_0,t_0} \}$. In the next section, we study the $L^2$ quasi-orthogonality of these functions, which we use as a starting point to obtain the required square root cancellation.

\section{Quasi-Orthogonality Estimates For Solutions to the Half-Wave Equation Obtained From High-Frequency Wave Packets} \label{estimatesforwavepackets}

The discussion at the end of Section \ref{PrelimSetup} motivates us to consider estimates for functions obtained by taking averages of the wave equation over a small time interval, with initial conditions localized to a particular part of space. In this section, we study the $L^2$ orthogonality of such quantities. We do not exploit periodicity of the Hamiltonian flow in this section since we are only dealing with estimates for the half wave-equation for \emph{small times}. Our results here thus hold for any manifold $M$, and any operator $P$ whose principal symbol has cospheres with non-vanishing Gaussian curvature.

In order to prove the required quasi-orthogonality estimates of the functions $\{ f_{x_0,t_0} \}$, we find a new connection between \emph{Finsler geometry} on the manifold $M$ and the behaviour of the operator $P$. We will not assume any knowledge of Finsler geometry in the sequel, describing results from the literature needed when required and trying to relate the tools we use to their Riemannian analogues. Moreover, in the case where $p(x,\xi) = \sqrt{\sum g^{jk}(x) \xi_j \xi_k}$ for some Riemannian metric $g$ on $M$, the Finsler geometry we study will simply be the Riemannian geometry on $M$ given by the metric $g$. We begin this section by briefly outlining the relevant concepts of Finsler geometry required to state Theorem \ref{theMainEstimatesForWave}, relegating more precise details to Subsection \ref{critpointsection} where Finsler geometry arises more explicitly in our proofs.


For our purposes, Finsler geometry is very akin to Riemannian geometry, though instead of a smoothly varying \emph{inner product} being given on the tangent spaces of $M$, in Finsler geometry we are given a smoothly varying \emph{vector space norm} $F: TM \to [0,\infty)$ on the tangent spaces of $M$. Many results in Riemannian geometry carry over to the Finsler setting. In particular, we can define the length of a curve $c: I \to M$ via it's derivative $c': I \to TM$ by the formula
\begin{equation}
    L(c) = \int_I F(c'),
\end{equation}
and thus get a theory of metric, geodesics, and curvature as in the Riemannian setting. The main quirk of the theory for us is that the metric induced from this length function is not symmetric. Indeed, let
\begin{equation}
    d_+(p_0,p_1) = \inf \big\{ L(c) : c(0) = p_0\ \text{and}\ c(1) = p_1 \big\},
\end{equation}
and set $d_-(p_0,p_1) = d_+(p_1,p_0)$, i.e. so that
\begin{equation}
    d_+(p_0,p_1) = \inf \big\{ L(c) : c(0) = p_1\ \text{and}\ c(1) = p_0 \big\}.
\end{equation}
In the Riemannian setting, $d_+$ and $d_-$ are both equal to the usual Riemannian metric, but on a general Finsler manifold one has $d_+ \neq d_-$, and these functions are distinct quasi-metrics on $M$ (though a compactness argument shows $d_M^+(x_0,x_1) \sim d_M^-(x_0,x_1)$). Geodesics from a point $p_0$ to a point $p_1$ need not be geodesics from $p_1$ to $p_0$ when reversed. A metric can be obtained by setting $d_M = (d_M^+ + d_M^-) / 2$, and we will use this as the canonical metric on $M$ in what follows.

Finsler geometry arises here because, if $p: T^* M \to [0,\infty)$ is a homogeneous function satisfying the curvature assumptions of Theorem \ref{CpVersionOfTheorem}, then a Finsler metric arises by taking the `dual norm' of $p$, i.e. setting
\begin{equation}
    F(x,v) = \sup \big\{ \xi(v) : \xi \in T^*_x M\ \text{and}\ p(x,\xi) = 1 \big\}.
\end{equation}
Thus $p$ induces quasimetrics $d_+$ and $d_-$ on $M$ via the metric $F$ as defined above. We now use these quasimetrics to state Proposition \ref{theMainEstimatesForWave}.

\pagebreak[3]

\begin{prop} \label{theMainEstimatesForWave}
    Let $M$ be a compact manifold of dimension $d$, and let $P$ be a classical elliptic self-adjoint pseudodifferential operator of order one whose principal symbol satisfies the curvature assumptions of Theorem \ref{CpVersionOfTheorem}. Then there exists $\varepsilon_M > 0$ such that for all $R \geq 0$, the following estimates hold:
    \begin{itemize}[leftmargin=8mm]
        \item (Pointwise Estimates) Fix  $|t_0| \leq \varepsilon_M$ and $x_0 \in M$. Consider any two measurable functions $c: \RR \to \CC$ and $u: M \to \CC$, with $\| c \|_{L^1(\RR)} \leq 1$ and $\| u \|_{L^1(M)} \leq 1$, and with $\supp(c) \subset I_{t_0}$ and $\supp(u) \subset B(x_0,1/R)$. Define $S: M \to \CC$ by setting
        \begin{equation}
            S = \int c(t) (Q_R \circ e^{2 \pi i t P} \circ Q_R) \{ u \}\; dt.
        \end{equation}
        Then for any $K \geq 0$, and any $x \in M$,
        \begin{equation}
            |S(x)| \lesssim_K \frac{R^d}{\langle R d_M(x_0,x) \rangle^{\frac{d-1}{2}}} \max\nolimits_{\pm} \Big\langle R \big| t_0 \pm d_M^\pm(x,x_0) \big| \Big\rangle^{-K}. 
        \end{equation}

        \item (Quasi-Orthogonality Estimates) Fix $|t_0 - t_1| \leq \varepsilon_M$, and $x_0, x_1 \in M$. Consider any two pairs of functions $c_0,c_1: \RR \to \CC$ and $u_0,u_1: M \to \CC$ such that, for each $\nu \in \{ 0, 1 \}$, $\| c_\nu \|_{L^1(\RR)} \leq 1$, $\| u_\nu \|_{L^1(M)} \leq 1$, $\supp(c_\nu) \subset I_{t_\nu}$, and $\supp(u_\nu) \subset B(x_\nu,1/R)$. Define $S_\nu: M \to \CC$ by setting
        \begin{equation}
            S_\nu = \int c_\nu(t) (Q_R \circ e^{2 \pi i t P} \circ Q_R) \{ u_\nu \}\; dt.
        \end{equation}
        Then for any $K \geq 0$,
        \begin{equation} \label{AWOICJAWOIVJEAO120321903120938}
            \left| \langle S_0, S_1 \rangle \right| \lesssim_K \frac{R^d}{\langle R d_M(x_0,x_1) \rangle^{\frac{d-1}{2}}} \max\nolimits_{\pm} \Big\langle R \big| (t_0 - t_1) \pm d^\pm(x_0,x_1) \big| \Big\rangle^{-K}.
        \end{equation}
    \end{itemize}
\end{prop}

%


\begin{remark}
Estimate \eqref{AWOICJAWOIVJEAO120321903120938} is a variable coefficient analogue of Lemma 3.3 of \cite{HeoandNazarovandSeeger}. The pointwise estimate tells us that the function $S$ is concentrated on a geodesic annulus of radius $|t_0|$ centered at $x_0$ and thickness $O(1/R)$. 
The quasi-orthogonality estimate tells us that the two functions $S_0$ and $S_1$ are only significantly correlated with one another if the two annuli on which the majority of the support of $S_0$ and $S_1$ lie are internally or externally tangent to one another, depending on whether $t_0$ and $t_1$ have the same or opposite sign respectively. 
%
%
\end{remark}

\begin{proof}[Proof of Proposition \ref{theMainEstimatesForWave}]

To simplify notation, in the following proof we will suppress the use of $R$ as an index, for instance, writing $Q$ for $Q_R$. For both the pointwise and quasi-orthogonality estimates, we want to consider the operators in coordinates, so we can use the \emph{Lax-H\"{o}rmander Parametrix} to understand the wave propagators in terms of various oscillatory integrals, though we use a slight variant of the usual parametrix which will help us in the stationary phase arguments which occur when manipulating the resulting oscillatory integrals.

Start by covering $M$ by a finite family of suitably small open sets $\{ V_\alpha \}$, such that for each $\alpha$, there is a coordinate chart $U_\alpha$ compactly containing $V_\alpha$ and with $N(V_\alpha, 1.1 \varepsilon_M) \subset U_\alpha$. Let $\{ \eta_\alpha \}$ be a partition of unity subordinate to $\{ V_\alpha \}$. It will be convenient to define $V_\alpha^* = N(V_\alpha, 0.01 \varepsilon_M )$ for each $\alpha$. The next lemma allows us to approximate the operator $Q$, and the propagators $e^{2\pi i t P}$ with operators which have more explicit representations in the coordinate system $\{ U_\alpha \}$, with an error negligible to the results of Proposition \ref{theMainEstimatesForWave}.


\pagebreak[3]

\begin{lemma} \label{pseudodifferentialCoordinateLemma}
    Suppose $\varepsilon_M$ is suitably small, depending on $M$ and $P$. For each $\alpha$, and $|t| \leq \varepsilon_M$, there exists Schwartz operators $Q_\alpha$ and $W_\alpha(t)$, each with kernels supported on $U_\alpha \times V^*_\alpha$, such that the following holds:
    \begin{itemize}[leftmargin=8mm]
    \setlength\itemsep{0.5em}
        \item For any $u: M \to \CC$ with $\| u \|_{L^1(M)} \leq 1$ and $\text{supp}(u) \subset V_\alpha^*$,
        \begin{equation}
            \text{supp}(Q_{\alpha} u) \subset N(\text{supp}(u), 0.01 \varepsilon_M),
        \end{equation}
        \begin{equation}
            \supp(W_{\alpha}(t) u) \subset N(\supp(u), 1.01 \varepsilon_M ),
        \end{equation}
        \begin{equation}
            \| (Q - Q_{\alpha}) u \|_{L^\infty(M)} \lesssim_N R^{-N} \quad\text{for all $N \geq 0$},
        \end{equation}
        and
        \begin{equation}
            \left\| \big( Q_{\alpha} \circ ( e^{2 \pi i t P} - W_{\alpha}(t) ) \circ Q_{\alpha} \big) \{ u \} \right\|_{L^\infty(M)} \lesssim_N R^{-N} \quad\text{for all $N \geq 0$}.
        \end{equation}

        \item In the coordinate system of $U_\alpha$, the operator $Q_{\alpha}$ is a pseudo-differential operator of order zero given by a symbol $\sigma_{\alpha}(x,\xi)$, where
        \begin{equation}
            \supp(\sigma_{\alpha}) \subset \{ \xi \in \RR^d: R/8 \leq |\xi| \leq 8R \},
        \end{equation}
        and $\sigma_{\alpha}$ satisfies derivative estimates of the form
        \begin{equation}
            |\partial^\lambda_x \partial^\kappa_\xi \sigma_{\alpha}(x,\xi)| \lesssim_{\lambda,\kappa} R^{-|\kappa|}.
        \end{equation}

        \item In the coordinate system $U_\alpha$, the operator $W_\alpha(t)$ has a kernel $W_\alpha(t,x,y)$ with an oscillatory integral representation
        \begin{equation}
            W_\alpha(t,x,y) = \int s(t,x,y,\xi) e^{2 \pi i [ \phi(x,y,\xi) + t p(y,\xi) ]}\; d\xi,
        \end{equation}
        where the function $s$ is compactly supported in $U_\alpha$, such that
        \begin{equation}
            \supp_{t,x,y}(s) \subset \{ (t,x,y) : |x - y| \leq C |t| \}
        \end{equation}
        for some constant $C > 0$, suc that
        \begin{equation}
            \supp_\xi(s) \subset \{ \xi \in \RR^d : R/8 \leq |\xi| \leq 8R \}, 
        \end{equation}
        and such that the function $s$ satisfies derivative estimates of the form
        \begin{equation}
            | \partial_{t,x,y}^\lambda \partial_\xi^\kappa s | \lesssim_{\lambda, \kappa} R^{- |\kappa|}.
        \end{equation}
        The phase function $\phi$ is smooth and homogeneous of degree one in the $\xi$ variable, and solves the \emph{eikonal equation}
        \begin{equation}
            p \left( x, \nabla_x \phi(x,y,\xi) \right) = p(y,\xi),
        \end{equation}
        subject to the constraint that $\phi(x,y,\xi) = 0$ if $\xi \cdot (x - y) = 0$.
        %
    \end{itemize}
\end{lemma}

We relegate the proof of Lemma \ref{pseudodifferentialCoordinateLemma} to the appendix, the proof being a technical and mostly conventional calculation involving a manipulation of oscillatory integrals using integration by parts.

Let us now proceed with the proof of the pointwise bounds in Proposition \ref{theMainEstimatesForWave} using Lemma \ref{pseudodifferentialCoordinateLemma}. Given $u: M \to \CC$, write $u = \sum_\alpha u_\alpha$, where $u_\alpha = \eta_\alpha u$. Lemma \ref{pseudodifferentialCoordinateLemma} implies that if we define
\begin{equation}
    S_\alpha = \int c(t) (Q_{j,\alpha} \circ W_{j,\alpha}(t) \circ Q_{j,\alpha}) \{ u_\alpha \}\; dt,
\end{equation}
then for all $N \geq 0$,
\begin{equation} \label{parametrixerrroestimate}
    \left\| S - \textstyle\sum\nolimits_\alpha S_\alpha \right\|_{L^\infty(M)} \lesssim_N R^{-N}.
\end{equation}
This error is negligible to the pointwise bounds we want to obtain in Proposition \ref{theMainEstimatesForWave} if we choose $N \geq K - \frac{d+1}{2}$, since the compactness of $M$ implies that $d_M(x,x_0) \lesssim 1$ for all $x \in M$, and so
\begin{equation}
    R^{-N} \lesssim R^{\left( \frac{d+1}{2} - K \right)} \lesssim \frac{R^{d}}{\langle R d_M(x,x_0) \rangle^{\frac{d-1}{2}}} \big\langle R \big| |t_0| \pm d_M^\pm(x,x_0) \big| \big\rangle^{-K}.
\end{equation}
We bound each of the functions $\{ S_\alpha \}$ separately, combining the estimates using the triangle inequality. We continue by expanding out the implicit integrals in the definition of $S_\alpha$. In the coordinate system $U_\alpha$, we can write
\begin{equation} \label{SalphaDefinition}
\begin{split}
    S_\alpha(x) &= \int c(t) \sigma(x,\eta) e^{2 \pi i \eta \cdot (x - y)}\\[-6 pt]
    &\quad\quad\quad s(t,y,z,\xi) e^{2 \pi i [ \phi(y,z,\xi) + t p(z,\xi) ]}\\
    &\quad\quad\quad\quad \sigma(z,\theta) e^{2 \pi i \theta \cdot (z - w)} (\eta_\alpha u)(w)\\
    &\quad\quad\quad\quad\quad dt\; dy\; dz\; dw\; d\theta\; d\xi\; d\eta.
\end{split}
\end{equation}
The integral in \eqref{SalphaDefinition} looks highly complicated, but can be simplified considerably by noticing that most variables are quite highly localized. In particular, oscillation in the $\eta$ variable implies that the amplitude is negligible unless $|x - y| \lesssim 1/R$, oscillation in the $\theta$ variable implies that the amplitude is negligible unless $|z - w| \lesssim 1/R$, and the support of $u$ implies that $|w - x_0| \lesssim 1/R$. Define
\begin{equation} \label{DIOAJVIOEJAV8318923}
    k_1(t,x,z,\xi) = \int \sigma(x,\eta) s(t,y,z,\xi) e^{2 \pi i [ \eta \cdot (x - y) + \phi(y,z,\xi) - \phi(x,z,\xi) ]}\; dy\; d\eta,
\end{equation}
and
\begin{equation} \label{DWAIOCWOAIJFAIO213123}
\begin{split}
    k_2(t,\xi) &= \int k_1(t,x,z,\xi) \sigma(z,\theta) (\eta_\alpha u)(w)\\
    &\quad\quad\quad e^{2 \pi i [ \theta \cdot (z - w) + \phi(x,z,\xi) - \phi(x,x_0,\xi) + t p(z,\xi) - t p(x_0,\xi) ]}\; d\theta\; dw,
\end{split}
\end{equation}
and then set
\begin{equation} \label{oaisdjoai9091390}
\begin{split}
    a(x,\xi) &= \int c(t) k_2(t,R \xi) e^{2 \pi i [ (t - t_0) p(x_0, R \xi) ]} \; dt\; dz,
\end{split}
\end{equation}
so that $\text{supp}_\xi(a) \subset \{ \xi : 1/8 \leq |\xi| \leq 8 \}$, and
\begin{equation}
    S_\alpha(x) = R^d \int a(x, \xi) e^{2 \pi i R [ \phi(x,x_0,\xi) + t_0 p(x_0,\xi) ]}\; d\xi.
\end{equation}
Integrating by parts in $\eta$ and $\theta$ in \eqref{DIOAJVIOEJAV8318923} and \eqref{DWAIOCWOAIJFAIO213123} gives that for all multi-indices $\alpha$,
\begin{equation} \label{ioqejfoieqjf13412}
    |\partial_\xi^\alpha k_1(t,x,z,\xi)| \lesssim_\alpha R^{-|\alpha|} \quad\text{and}\quad |\partial_\xi^\alpha k_2(z,\xi)| \lesssim_{\alpha} R^{-|\alpha|}.
\end{equation}
Using the bounds in \eqref{ioqejfoieqjf13412} with the fact that $\text{supp}(c)$ is contained in a $O(1/R)$ neighborhood of $t_0$ in \eqref{oaisdjoai9091390} then implies $|\partial_\xi^\alpha a(x,\xi)| \lesssim_\alpha 1$ for all $\alpha$.
%
%
%
%
%

We now account for angular oscillation of the integral by working in a kind of 'polar coordinate' system. First we find $\lambda: V_\alpha^* \times S^{d-1} \to (0,\infty)$ such that for all $|\xi| = 1$,
\begin{equation}
    p(x_0, \lambda(x_0,\xi) \xi) = 1.
\end{equation}
If $\tilde{a}(x,\rho, \eta) = a(x, \rho \lambda(x_0) \xi ) \det [ \lambda(x_0,\xi) I + \xi (\nabla_\xi \lambda)(x_0,\xi)^T ]$, then
%
%
\begin{equation}
    S_\alpha(x) = R^d \int_0^\infty \rho^{d-1} \int_{|\xi| = 1} \tilde{a}(x,\rho, \xi) e^{2 \pi i R \rho [ t_0 + \phi(x, x_0, \lambda(x_0,\xi) \xi) ]}\; d\xi\; d\rho.
\end{equation}
Define $\Phi: S^{d-1} \to \RR$ by setting $\Phi(\xi) = \phi(x,x_0, \lambda(x_0,\xi) \xi)$.    
We claim that, in the $\xi$ variable, $\Phi$ has exactly two critical points $|\xi^+|^{-1} \xi^+$ and $|\xi^-|^{-1} \xi^-$, where $\xi^+ \in S_{x_0}^*$ is the covector corresponding to the forward geodesic from $x_0$ to $x$, and $\xi^- \in S_{x_0}^*$ is the covector corresponding to the backward geodesic from $x_0$ to $x$. Moreover,
\begin{equation}
    \Phi(|\xi^+|^{-1} \xi^+) = d_M^+(x_0,x) \quad\text{and}\quad \Phi(|\xi^-|^{-1} \xi^-) = - d_M^-(x_0,x),
\end{equation}
and the Hessian at each of these points is non-degenerate, with each eigenvalue of the Hessian having magnitude exceeding a constant multiple of $d_M^{\pm}(x_0,x)$. We prove that these properties hold for $\Phi$ in Proposition \ref{triangleLemma} of the following section, via a series of geometric arguments. It then follows from the principle of stationary phase that
\begin{equation}
    S_\alpha(x) = \sum_{\pm} \frac{R^{d}}{\left\langle R d_M^{\pm}(x_0,x) \right\rangle^{\frac{d-1}{2}}} \int_0^\infty \rho^{\frac{d-1}{2}} a_{\pm}(x,\rho) e^{2 \pi i R \rho [ t_0 \pm d_M^{\pm}(x_0,x)]}\; d\rho,
\end{equation}
where $a_{\pm}$ is supported on $|\rho| \sim 1$, and for all $\alpha$, $|\partial_\rho^\alpha a_{\pm}| \lesssim_\alpha 1$. Integrating by parts in the $\rho$ variable if $t_0 \pm d_M^{\pm}(x_0,x)$ is large, we conclude that
%
%
%
%
%
\begin{equation} \label{finaloscillatoryintegralbound}
\begin{split}
    |S_\alpha(x)| \lesssim \frac{R^{d}}{\left\langle R d_M(x_0,x) \right\rangle^{\frac{d-1}{2}}} \sum_{\pm} \big\langle R |t_0 \pm d_M(x_0,x)| \big\rangle^{-K}.
\end{split}
\end{equation}
Combining \eqref{parametrixerrroestimate} and \eqref{finaloscillatoryintegralbound} completes the proof of the pointwise bounds.

The quasi-orthogonality arguments are obtained by a largely analogous method, and so we only sketch the proof. One major difference is that we can use the self-adjointness of the operators $Q$, and the unitary group structure of $\{ e^{2 \pi i t P} \}$, to write
\begin{equation}
\begin{split}
    \langle S_0, S_1 \rangle &= \int c_0(t) c_1(s) \big\langle (Q \circ e^{2 \pi i t P} \circ Q) \{ u_0 \}, (Q \circ e^{2 \pi i s P} \circ Q) \{ u_1 \} \big\rangle\\
    &= \int c_0(t) c_1(s) \big\langle (Q^2 \circ e^{2 \pi i (t - s) P} \circ Q^2) \{ u_0 \}, u_1 \big\rangle\\
    &= \int c(t) \big\langle (Q^2 \circ e^{2 \pi i t P} \circ Q^2) \{ u_0 \}, u_1 \big\rangle,
\end{split}
\end{equation}
where $c(t) = \int c_0(u) c_1(u - t)\; du$, by Young's inequality, satisfies
\begin{equation}
    \| c \|_{L^1(\RR)} \lesssim \| c_0 \|_{L^1(\RR)} \| c_1 \|_{L^1(\RR)} \leq 1
\end{equation}
and $\supp(c) \subset [ (t_0 - t_1) - 4/R, (t_0 - t_1) + 4/R]$. After this, one proceeds exactly as in the proof of the pointwise estimate. We write the inner product as
\begin{equation}
    \sum\nolimits_\alpha \int c(t) \big\langle (Q^2 \circ e^{2 \pi i t P} \circ Q^2) \{ \eta_\alpha u_0 \}, u_1 \big\rangle.
\end{equation}
Then we use Lemma \ref{pseudodifferentialCoordinateLemma} to replace $Q^2 \circ e^{2 \pi i tP} \circ Q^2$ with $Q_{\alpha}^2 \circ W_{\alpha}(t) \circ Q_{\alpha}^2$ using Lemma \ref{pseudodifferentialCoordinateLemma}, modulo a negligible error. The integral
\begin{equation}
    \sum\nolimits_\alpha \int c(t) \big\langle (Q_{\alpha}^2 \circ W_{\alpha}(t) \circ Q_{\alpha}^2) \{ \eta_\alpha u_0 \}, u_1 \big\rangle
\end{equation}
is then only non-zero if both the supports of $u_0$ and $u_1$ are compactly contained in $U_\alpha$. Thus we can switch to the coordinate system of $U_\alpha$, in which we can express the inner product by oscillatory integrals of the exact same kind as those occurring in the pointwise estimate. Integrating away the highly localized variables as in the pointwise case, and then applying stationary phase in polar coordinates proves the required estimates.
\end{proof}

\subsection{Some Facts About Finsler geometry} \label{BriefIntroduction}

All that remains to conclude the proof of the quasi-orthogonality estimates is to verify the required properties of the critical points of the phase function $\Phi$ which occurs in the proof of Proposition \ref{theMainEstimatesForWave}. To do this we exploit the Finsler geometry on the manifold, and so before we carry out this task we take out the time to precisely introduce the concepts of Finsler geometry that occur in the proof. We will work in a particular coordinate system $U_0$, so we introduce the concepts in coordinates to keep things concrete.


Let $U_0 \subset \RR^d$ be an open set. Then we may identify the tangent space $T U_0$ and cotangent space $T^* U_0$ with $U_0 \times \RR^d$, and each fibre $T_x U_0$ and $T^*_x U_0$ with $\RR^d$, though as is standard in differential we use upper indices to denote the coordinates of vectors, and lower indices to denote the coordinates of covectors. A \emph{Finsler metric} on $U_0$ is a homogeneous function $F: T U_0 \to [0,\infty)$ which is smooth on $TU_0 - 0$, and strictly convex, in the sense that the Hessian matrix with entries
\begin{equation} \label{FinslerMetricCoefficients}
    g_{ij}(x,v) = \frac{1}{2} \frac{\partial^2 F^2}{\partial v^i \partial v^j}(x,v)
\end{equation}
is positive definite for all $x \in U_0$ and $v \in T_x U_0 - \{ 0 \}$.


Given a metric $F$ on $M$, we define a dual metric $F_*: T^*M \to [0,\infty)$ by setting $F_*(x,\xi) = \sup \{ \xi(v) : F(x,v) = 1 \}$. Then $F_*$ is also strictly convex, so that for each $x \in U_0$ and each covector $\xi \in \RR^d - \{ 0 \}$, the Hessian matrix with entries
\begin{equation}
    g^{ij}(x,\xi) = \frac{1}{2} \frac{\partial^2 F_*^2}{\partial \xi_i \partial \xi_j}(x,\xi)
\end{equation}
is positive definite. For $x \in U_0$, the \emph{Legendre transform} is then defined to the smooth map $\mathcal{L}_x: T_x U_0 - \{ 0 \} \to T^*_x U_0 - \{ 0 \}$ given by the formula $(\mathcal{L}_x v)_i = \sum_j g_{ij}(x,v) v^j$, defined so that $\sum (\mathcal{L}_x v)_j w^j = \sum g_{ij}(x,v) v^j w^j$ for any vector $w \in \RR^d$. Then the matrix with entries $g_{ij}(x,v)$ is the inverse of the matrix with entries $g^{ij}(x, \mathcal{L}_x v)$, so that for a covector $\xi \in \RR^d - \{ 0 \}$, $(\mathcal{L}_x^{-1} \xi)^i = \sum g^{ij}(x,\xi) \xi_j$. The Legendre transform is the Finsler variant of the musical isomorphism in Riemannian geometry, though the Legendre transform is in general only \emph{homogeneous}, rather than linear.

Let us prove in more detail that if $p: T^*M \to [0,\infty)$ is homogenous and satisfies the assumptions of Theorem \ref{CpVersionOfTheorem}, the function $F(x,v) = \sup \{ \xi(v): p(x,\xi) = 1 \}$ gives a Finsler metric on $M$. Working in a coordinate system, we may assume $M = U_0$ is an open subset of $\RR^d$ like above. For each $x \in U_0$, the cosphere $S_x^* = \{ \xi \in T_x^* M : p(x,\xi) = 1 \}$ has non-vanishing Gaussian curvature. Thus all principal curvatures of $S_x^*$ must actually be \emph{positive}, since $S_x^*$ is a compact hypersurface in $T^*_x M$. This follows from a simple modification of an argument found in Chapter 2 of \cite{HeinzHopf}. Indeed, if we fix an arbitrary point $v_0 \in T_x^*M$, and consider the smallest closed ball $B \subset T_x^* M$ centered at $v_0$ and containing $S_x^*$, then the sphere $\partial B$ must share the same tangent plane as $S_x^*$ at some point. All principal curvatures of $\partial B$ are positive, and at this point all principal curvatures of $S_x^*$ must be greater than the principal curvatures of $\partial B$, since $S_x^*$ curves away faster than $\partial B$ in all directions. By continuity, we conclude that the principal curvatures are everywhere positive.  Thus for each $x \in M$ and $\xi \in T_x^* M - \{ 0 \}$, the coefficients $g^{ij}(x,\xi) = (1/2) (\partial^2 p^2 / \partial \xi_i \partial \xi_j)$ form a positive-definite matrix. But inverting the procedure of the last paragraph shows that the dual norm $F(x,v) = \sup\nolimits_{\xi \in S_x^*} \xi(v)$ is also strictly convex, and thus gives a Finsler metric on $M$ with $F_* = p$. 

The theory of geodesics on Finsler manifolds is similar to the Riemannian case, except for the interesting quirk that a geodesic from a point $p$ to a point $q$ need not necesarily be a geodesic when considered as a curve from $q$ to $p$, and so we must consider \emph{forward} and \emph{backward} geodesics. We define the length of a curve $c: I \to U_0$ by the formula $L(c) = \int_I F(c,\dot{c})\; dt$, and use this to define the \emph{forward distance} $d_+: U_0 \times U_0 \to [0,\infty)$ by taking the infima of paths between points. This function is a \emph{quasi-metric}, as it is not necessarily symmetric. We define the \emph{backward distance} $d_-: U_0 \times U_0 \to [0,\infty)$ by setting $d_-(p,q) = d_+(q,p)$. A \emph{constant speed speed forward geodesic} on a Finsler manifold is a curve $c: I \to U_0$ which satisfies the geodesic equation
%
%
\begin{equation} \label{geodesicequation}
    \ddot{c}^a = - \sum\nolimits_{j,k} \gamma^a_{jk}(c,\dot{c}) \dot{c}^j \dot{c}^k \quad\text{for $1 \leq i \leq d$},
\end{equation}
where $\gamma^a_{jk} = \sum_i (g^{ai}/2) ( \partial_{x_k} g_{ij} + \partial_{x_j} g_{ik} - \partial_{x_i} g_{jk} )$ are the \emph{formal Christoffel symbols}. 
%
%
We call the reversal of such a geodesic a \emph{backward geodesic}. A curve $c$ is a geodesic if and only if it's Legendre transform $(x,\xi) = \mathcal{L}(c,\dot{c})$ satisfies
\begin{equation} \label{FStarHamilton}
    x'(t) = (\partial_\xi F_*)(x,\xi) \quad\text{and}\quad \xi'(t) = - (\partial_x F_*)(x,\xi).
\end{equation}
The equations \eqref{FStarHamilton} are first order ordinary differential equations induced by the Hamiltonian vector field $( \partial_\xi F_*, - \partial_x F_* )$ on $T^* U_0$. Note that in our situation, the dual norm $F_*$ is the principal symbol $p$ of the operator $P$ we are studying, and so \eqref{FStarHamilton} is precisely the Hamiltonian flow alluded to in \eqref{qoiJIOJdoiwajfoiawjfoi}. Sufficiently short geodesics on a Finsler manifold are length minimizing. In particular, there exists $r > 0$ such that if $c$ is a geodesic between two points $x_0$ and $x_1$, and $L(c) < r$, then $d_M^+(x_0,x_1) = L(c)$.

%
%
%
%
%

The last piece of technology we must consider before our analysis of critical points are \emph{variation formulas} for arclength. Consider a function $A: (-\varepsilon,\varepsilon) \times [0,1] \to U_0$, such that for each $s$, the function $t \mapsto A(s,t)$ is a constant speed geodesic, and such that all such geodesics begin at the same point, i.e. so that $A(s,0)$ is independent of $s$. Define the vector field $V(s) = (\partial_s A)(s,1)$ and $T(s) = (\partial_t A)(s,1)$. Then if $R(s)$ gives the arclength of the geodesic $t \mapsto A(s,t)$, the first variation formula for energy tells us
\begin{equation} \label{FirstVariationFormula2}
    R(s) R'(s) = \sum\nolimits_{i,j} g_{ij}( A(s,1), T(s) ) V^i(s) T^j(s).
\end{equation}
Details can be found in Chapter 5 of \cite{BaoChern}. We will consider \emph{second variations}, i.e. estimates on $R''$, but we will not need the second variation formula to obtain the estimates we need, relying only on differentiating the equation \eqref{FirstVariationFormula2} rather than explicitly working with Jacobi fields. 
%

\begin{remark}
    The second variation formula for arclength in Riemannian geometry has an exact analogue in Finsler geometry via a generalization of the theory of Jacobi fields. Sectional curvature is replaced with what is called \emph{flag curvature}, which controls the growth of Jacobi fields. One could in principle use such quantities to obtain more explicit constants in Theorem \ref{triangleLemma}, though we have no such need of these constants here.
\end{remark}

\subsection{Analysis of Critical Points} \label{critpointsection}

In this section, we now classify the critical points of the kinds of functions that arose in Proposition \ref{theMainEstimatesForWave}.

\pagebreak[3]

\begin{prop} \label{triangleLemma}
    Fix a bounded open set $U_0 \subset \RR^d$. Consider a Finsler metric $F: U_0 \times \RR^d \to [0,\infty)$ on $U_0$, and it's dual metric $F_*: U_0 \times \RR^d \to [0,\infty)$, which extends to a Finsler metric on an open set containing the closure of $U_0$. Fix a suitably small constant $r > 0$. Let $U$ be an open subset of $U_0$ with diameter at most $r$ which is geodesically convex (any two points are joined by a minimizing geodesic). Let $\phi: U \times U \times \RR^d \to \RR$ solve the eikonal equation
    \begin{equation} \label{eikonalequation}
        F_* ( x, \nabla_x \phi(x,y,\xi) ) = F_*(y,\xi),
    \end{equation}
    such that $\phi(x,y,\xi) = 0$ for $x \in H(y,\xi)$, where $H(y, \xi) = \{ x \in U : \xi \cdot (x - y) = 0 \}$. For each $x_0,x_1 \in U$, let $S_{x_0}^* = \{ \xi \in \RR^d: F_*(x_0,\xi) = 1 \}$ be the cosphere at $x_0$, and define $\Psi: S_{x_0}^* \to \RR$ by setting
    \begin{equation}
        \Psi(\xi) = \phi(x_1,x_0,\xi).
    \end{equation}
    Then the function $\Psi$ has exactly two critical points, at $\xi^+$ and $\xi^-$, where the Legendre transform of $\xi^+$ is the tangent vector of the forward geodesic from $x_0$ to $x_1$, and the Legendre transform of $\xi^-$ is the tangent vector of the backward geodesic from $x_1$ to $x_0$. Moreover,
    \begin{equation}
        \Psi(\xi^+) = d_M^+(x_0,x_1) \quad\text{and}\quad \Psi(\xi^-) = - d_M^-(x_0,x_1),
    \end{equation}
    and the Hessians $H_+$ and $H_-$ of $\Psi$ at these critical points, viewed as quadratic maps from $T_{\xi_{\pm}} S_{x_0}^* \to \RR$ satisfy
    \begin{equation}
        H_+(\zeta) \geq C d_M^+(x_0,x_1) |\zeta| \quad\text{and}\quad H_-(\zeta) \leq - C d_M^-(x_0,x_1) |\zeta|,
    \end{equation}
    where the implicit constant is uniform in $x_0$ and $x_1$.
\end{prop}

If $\Psi$ is as above, then the function $\Phi: S^{d-1} \to \RR$ obtained by setting $\Phi(\xi) = \phi( x, x_0, F_*(x,\xi)^{-1} \xi  )$ is precisely the kind of function that arose as a phase in Proposition \ref{theMainEstimatesForWave}, where $F_*$ was the principal symbol $p$ of the pseudodifferential operator we were considering. Since critical points and the Hessians of maps at critical points are stable under diffeomorphisms, 
and the map $\xi \mapsto F_*(x,\xi)^{-1} \xi$ is a diffeomorphism from $S_{x_0}^*$ to $S^{d-1}$, classifying the critical points of the map $\Psi$ implies the required properties of the map $\Phi$ used in Proposition \ref{theMainEstimatesForWave}.

In order to prove \ref{triangleLemma}, we rely on a geometric interpretation of $\Psi$ following from Hamilton-Jacobi theory.

\begin{lemma} \label{HamiltonLemma}
    Consider the setup to Proposition \ref{triangleLemma}. For any $\xi \in S_{x_0}^*$,
    \[ |\Psi(\xi)| = \begin{cases} \text{the length of the shortest curve from $H(x_0,\xi)$ to $x_1$} & : \text{if}\ \Psi(\xi) > 0, \\ \text{the length of the shortest curve from $x_1$ to $H(x_0,\xi)$} & : \text{if}\ \Psi(\xi) < 0. \end{cases} \]
\end{lemma}
\begin{proof}
    We rely on a construction of $\phi$ from Proposition 3.7 of \cite{Treves2}, which we briefly describe. Fix $x_0$ and $\xi$. Then there is a unique covector field $\omega: H(x_0,\xi) \to \RR^d$ which is everywhere perpendicular to $H(x_0,\xi)$, with $\omega(x_0) = \xi$ and with $F_*(x,\omega(x)) = F_*(x_0,\xi)$ for all $x \in H(x_0,\xi)$. There exists a unique point $x(\xi) \in H(x_0,\xi)$ and a unique $t(\xi) \in \RR$, such that the unit speed geodesic $\gamma$ on $M$ with $\gamma(0) = x(\xi)$ and $\gamma'(0) = \mathcal{L}^{-1}( x(\xi), \omega(x(\xi)) )$ satisfies $\gamma( t(\xi) ) = x_1$. We then have $\Psi(\xi) = t(\xi)$. If $t(\xi)$ is negative, then $\gamma|_{[t(\xi),0]}$ is a geodesic from $x_1$ to $x_0$, and if $t(\xi)$ is positive, $\gamma|_{[0,t(\xi)]}$ is a geodesic from $x_0$ to $x_1$. Because $\gamma$ is a geodesic, the geometric interpretation then follows if $U$ is a suitably small neighborhood such that geodesics are length minimizing.
\end{proof}

It follows immediately from Lemma \ref{HamiltonLemma} that
\begin{equation}
    \Psi(\xi_+) = d_M^+(x_0,x_1) \quad\text{and}\quad \Psi(\xi_-) = - d_M^-(x_0,x_1).
\end{equation}
A simple geometric argument also shows $\Psi(\xi^+)$ is the maximum value of $\Psi$ on $S_{x_0}^*$, and $\Psi(\xi^-)$ is the minimum value on $S_{x_0}^*$, so that these two points are both critical. Indeed, the point $x_0$ lies in $H(x_0,\xi)$ for all $\xi \in \RR^d - \{ 0 \}$. Thus the shortest curve from $x_0$ to $x_1$ is always longer than the shortest curve from $H(x_0,\xi)$ to $x_1$. Similarily, the shortest curve from $x_1$ to $x_0$ is always longer than the shortest curve from $x_1$ to $H(x_0,\xi)$. Thus
\begin{equation}
    - d_M^-(x_0,x_1) \leq \Psi(\xi) \leq d_M^+(x_0,x_1),
\end{equation}
and so $\Psi(\xi^-) \leq \Psi(\xi) \leq \Psi(\xi^+)$ for all $\xi \in $. All that remains is to prove that $\xi^+$ and $\xi^-$ are the \emph{only} critical points of $\Psi$, and that these critical points are appropriately non-degenerate.


In order to simplify proofs, we employ a structural symmetry to reduce the number of cases we need to analyze. Namely, if one defines the reverse Finsler metric $F_\rho(x,v) = F(x,-v)$, then $F_\rho^*(x,\xi) = F_*(x,-\xi)$, and so the associated function $\Psi^\rho$ which is the analogue of $\Psi$ for $F^\rho$ satisfies $\Psi^\rho(\xi) = -\Psi(-\xi)$. The critical points of $\Psi$ and $\Psi^\rho$ are thus directly related to one another, which allows us without loss of generality to study only points with $\Psi \leq 0$ (and thus only study geodesics beginning at $x_1$).

\begin{proof}[Proof that $\xi^+$ and $\xi^-$ are the only critical points]
    Fix $\xi^* \in T_{x_0} M - \{ \xi^{\pm} \}$. Using the notation defined above, let $x_* = x(\xi^*)$ and $t_* = t(\xi^*)$. Using the symmetry above, we may assume without loss of generality that $\Psi(\xi^*) \leq 0$. Since $\xi^-$ is not perpendicular to $H(x_0,\xi^*)$ at $x_0$, we have $x_* \neq x_0$. If $\Psi(\xi^*) < 0$, let $\gamma$ be the unique unit speed forward geodesic with $\gamma(0) = x_1$ and $\gamma(t_*) = x_*$. If $\Psi(\xi^*) = 0$, let $\gamma$ be the unique unit speed forward geodesic with $\gamma'(0)$ equal to the Legendre transform of $\xi^*$. Pick $\eta$ such that $\eta \cdot (x_* - x_0) \neq 0$, and then, for $t$ suitably close to $t_*$, define a smooth map
    \begin{equation} \xi(t) = \frac{\xi^* + a(t) \eta}{F_*(x_0, \xi^* + a(t) \eta)}, \end{equation}
    into $S_{x_0}^*$, where
    \begin{equation} a(t) = - \frac{\xi^* \cdot ( \gamma(t) - x_0 )}{\eta \cdot ( \gamma(t) - x_0 )} \end{equation}
    is defined so that $\gamma(t) \in H(x_0, \xi(t))$. Then $\Psi( \xi(t) ) = -t$, and differentiation at $t = t_*$ gives $D\Psi( \xi^* ) ( \xi'(t_*) ) = -1$. In particular, $D \Psi( \xi^* ) \neq 0$, so $\xi^*$ is not a critical point. \qedhere
    %
    %
    %
    %


    %
    %

    \end{proof}

We now analyze the non-degeneracy of the critical points $\xi^+$ and $\xi^-$.

\begin{proof}[Proof that $\xi^+$ and $\xi^-$ are non-degenerate]
    Using symmetry, it suffices without loss of generality to analyze the critical point $\xi^-$ rather than $\xi^+$. Let $H_-: T_{\xi^-} S_{x_0}^* \to \RR$ be the Hessian of $\Psi$ at $\xi^-$. Let $v_0 = \mathcal{L}_{x_0}^{-1} \xi^-$, and using the notation of the last argument, let $l = t(\xi^-)$. Consider a curve $\xi(a)$ valued in $S_{x_0}^*$ with $\xi(0) = \xi^-$ and $\zeta = \xi'(0)$ for some $\zeta \in T_{\xi^-}^* S_{x_0}$. Then the second derivative of $\Psi(\xi(a))$ at $a = 0$ is $H_-( \zeta )$, and so our proof would be complete if we could show that the function $L(a) = - \Psi( \xi(a) )$, which is the length of the shortest geodesic from $x_1$ to $H(x_0,\xi(a))$, satisfies $L''(0) \geq C l |\zeta|$, for a constant $C > 0$ uniform in $x_0$, $x_1$, and $\zeta$.

    Consider the partial function $\text{Exp}_{x_1}: \RR^d \to U_0$ obtained from the exponential map at $x_1$. If $r$ is chosen suitably small, there exists a neighborhood $V$ of the origin in $\RR^d$ such that $E = \text{Exp}_{x_1}|_V$ is a diffeomorphism between $V$ and $U$. Unlike in Riemannian manifolds, in general $E$ is only $C^1$ at the origin, though smooth on $V - \{ 0 \}$. However, for $|v| = 1$ and $t > 0$ we can write $E(tv) = (\pi \circ \varphi)( x_1, \mathcal{L}_{x_1} v, t)$, where the partial function $\varphi: (T^* U_0 - 0) \times \RR \to (T^* U_0 - 0)$ is the flow induced by the the Hamiltonian vector field $( \partial_\xi F_*, - \partial_x F_*)$ on $T^* U_0 - 0$, and $\pi$ is the projection map from $T^* U_0 - 0$ to $U_0$. Where defined, $\varphi$ is a smooth function since the Hamiltonian vector field is smooth, and so it follows by homogeneity and the precompactness of $U$ that the partial derivatives $(\partial^\alpha E_x)(v)$ are uniformly bounded for $v \neq 0$ and $x \in U$. It thus follows from the inverse function theorem that there exists a constant $A > 0$ such that for all $x_1$ and $x$ in $U$ with $x \neq x_1$,
    \begin{equation} \label{GBounds}
        |\partial_j G_{x_1}(x)| \leq A \quad\text{and}\quad |(\partial_j \partial_k G_{x_1})(x)| \leq A.
    \end{equation}
    We can also pick $A$ to be large enough that for all $x \in U$, and all $v,w \in \RR^d - \{ 0 \}$,
    \begin{equation}
        A^{-1} |w|^2 \leq \sum\nolimits_{ij} g_{ij}(x,v) w^i w^j \leq A |w|^2.
    \end{equation}
    and such that for all $x \in U$ and $v \in \RR^d - \{ 0 \}$,
    \begin{equation} \label{gijbounds}
        |\partial_{x_k} g_{ij}(x,v)| \leq A.
    \end{equation}
    %

    Since $\zeta \in T_{\xi^-} S_{x_0}^*$, and $S_{x_0}^* = \{ \xi : F_*(x_0,\xi)^2 = 1 \}$, by Euler's homogeneous function theorem we have
    \begin{equation} \label{DIOAWJDIOWAJDOIWAJ14}
        \sum\nolimits_{ij} g^{ij}(x_0,\xi^-) \xi'(0) \xi^-_j = \frac{1}{2} \sum\nolimits_{i,j} \frac{\partial^2 F_*^2}{\partial \xi_i \partial \xi_j}(x_0,\xi^-) \zeta_i \xi^-_j = \frac{1}{2} \sum \frac{\partial F_*^2}{\partial \xi_j}(x_0,\xi^-) \zeta_j = 0.
    \end{equation}
    Differentiating $F_*(x_0,\xi(a))^2 = 1$ twice with respect to $a$ at $a = 0$ yields that
    \begin{equation} \label{AIWODJWAO314213}
         \sum\nolimits_{i,j} g_{ij}(x_0,\xi^-) \xi''_i(0) \xi^-_j = - \sum\nolimits_{i,j} g_{ij}(x_0,\xi^-) \zeta^i \zeta^j
    \end{equation}
    Define vectors $n(a)$ by setting $n^i(a) = \sum g^{ij}(x_0,\xi^-) \xi_j(a)$. Then $n(a)$ is the normal vector to $H(x_0, \xi(a))$ with respect to the inner product with coefficients $g_{ij}(x_0,v_0)$.
    %
    %
    Let $u(a)$ be the orthogonal projection of $v_0$ onto the hyperplane $\{ v : \xi(a) \cdot v = 0 \}$ with respect to the inner product $g_{ij}(x_0,v_0)$. If we define
    \begin{equation}
        c(a) = \sum g_{ij}(x_0,v_0) v_0^i n^j(a) = \sum g^{ij}(x_0,\xi^-) \xi^-_i \xi_j(a),
    \end{equation}
    then $u(a) = v_0 - c(a) n(a)$. Note that \eqref{DIOAWJDIOWAJDOIWAJ14} and \eqref{AIWODJWAO314213} imply $c(0) = 1$, $c'(0) = 0$, and $c''(0) \leq - |\zeta|^2 / A$. Also $n(0) = v_0$ and $|n'(0)| \leq A |\zeta|$.
    %

    Let $x(s) = x_0 + s u(a)$ and let $R(s)$ be the length of the geodesic from $x_1$ to $x(s)$. To control $R(s)$, define $y(s) = G(x(s))$, and consider the variation $A(s,t) = E( t y(s) )$, defined so that $t \mapsto A(s,t)$ is the geodesic from $x_1$ to $x(s)$. The Gauss Lemma for Finsler manifolds (see Lemma 6.1.1 of \cite{BaoChern}) implies that
    \begin{equation}
        A^{-1} R(s) \leq |y(s)| \leq A R(s).
    \end{equation}
    Define $T(s) = (\partial_t A)(s,1)$ and $V(s) = (\partial_s A)(s,1)$.
    %
    %
    Then \eqref{FirstVariationFormula2} implies
    \begin{equation} \label{RSFirstDerivativeEquation}
        R(s) R'(s) = \sum\nolimits_{i,j} g_{ij}( x(s), T(s) ) V^i(s) T^j(s).
    \end{equation}
    Again, the Gauss Lemma implies
    \begin{equation} \label{Idiawjdiwaj213123}
        A^{-1} R(s) \leq |T(s)| \leq A R(s).
    \end{equation}
    We can write
    \begin{equation}
        T^i(s) = \sum\nolimits_{i,j} (\partial_j E^i)( y(s)) y^j(s)
    \end{equation}
    and
    \begin{equation}
        V^i(s) = \sum\nolimits_{i,j,k} (\partial_j E^i) (y(s)) (\partial_k G^j)( x(s)) u^k(a) = u^i(a).
    \end{equation}
    In particular, $T(0) = v_0$, so $R'(0) = \sum g_{ij}(x_0, v_0) u^i(a) v_0^j = 1 - c(a)^2$. Cauchy-Schwartz applied to \eqref{RSFirstDerivativeEquation} also tells us that
    \begin{equation} \label{wdkawoidj141}
        |R'(s)| \lesssim_A |u(a)|.
    \end{equation}
    Note that $u(0) = 0$, so if $a$ is small enough, then we have $|u(a)| \leq l/2$. Taylor's theorem applied to \eqref{wdkawoidj141}, noting $R(0) = l$ then gives that for $|s| \leq 1$,
    \begin{equation} \label{RBound}
        l/2 \leq R(s) \leq 2l.
    \end{equation}
    Differentiating \eqref{RSFirstDerivativeEquation} tells us that
    \begin{equation} \label{FFFF213123}
    \begin{split}
        R'(s)^2 + R''(s) R(s) &= \sum\nolimits_{i,j,k} \Bigg[ (\partial_{x_k} g_{ij})( x(s), T(s) ) u^i(a) T^j(s) u^k(a) \Bigg]\\
        &\quad\quad\quad\quad + \Bigg[ (\partial_{v_k} g_{ij} )(x(s), T(s)) u^i(a) T^j(s) (\partial_s T^k)(s)   \Bigg]\\
        &\quad\quad\quad\quad\quad\quad + \Bigg[ g_{ij}( x(s), T(s) ) u^i(a) (\partial_s T^j)(s) \Bigg].
    \end{split}
    \end{equation}
    Write the right hand side as $\text{I} + \text{II} + \text{III}$. Using \eqref{gijbounds}, \eqref{Idiawjdiwaj213123}, \eqref{RBound}, and the triangle inequality gives
    \begin{equation} \label{IBound}
        |\text{I}| \lesssim R(s) |u(a)|^2 \lesssim l |u(a)|^2.
    \end{equation}
    Since $\partial_{v_k} g_{ij} = (1/2) (\partial^3F^2 / \partial v_i \partial v_j \partial v_k)$, applying Euler's homogeneous function theorem when summing over $j$ implies that
    \begin{equation} \label{IIBound}
        \text{II} = 0.
    \end{equation}
    Applying Cauchy-Schwarz and \eqref{GBounds}, we find
    \begin{equation} \label{IIIBound}
        |\text{III}| \lesssim |u(a)| |T'(s)| \lesssim |u(a)|^2 [|y(s)| + 1] \lesssim (l + 1) |u(a)|^2.
    \end{equation}
    But now combining \eqref{IBound}, \eqref{IIBound}, \eqref{IIIBound}, and rearranging \eqref{FFFF213123} shows that
    \begin{equation}
        |R''(s)| \lesssim l^{-1} |u(a)|^2 + (1 + l^{-1}) |u(a)|^2 \lesssim l^{-1} |u(a)|^2.
    \end{equation}
    Taylor's theorem implies there exists $B > 0$ depending only on $A$ and $d$ such that
    \begin{equation} \label{WADAWD21312}
        |R(s) - (R(0) + s R'(0))| \leq B l^{-1} |u(a)|^2 s^2.
    \end{equation}
    Since $R(0) = $l, $R'(0) = 1 - c(a)^2$, $c(0) = 1$, $u(0) = 0$, $c'(0) = 0$, $|u'(0)| \leq A |\zeta|$, and $c''(0) \leq - |\zeta|^2 / A$, we conclude from \eqref{WADAWD21312} that as $a \to 0$, if $s > 0$ then
    \begin{equation}
    \begin{split}
        R(-s) &\leq l - s R'(0) + B l^{-1} |u(a)|^2 s^2\\
        &\leq l - s ( A^{-1} |\zeta|^2 a^2 ) + A^2 B l^{-1} |\zeta|^2 s^2 a^2 + O(a^3 ( s + l^{-1} s^2)).
    \end{split}
    \end{equation}
    %
    For all $s$, $L(a) \leq R(s)$. Optimizing by picking $s = l / 2 A^3 B$ gives
    \begin{equation}
        L(a) \leq R(-s) \leq l - l |\zeta|^2 a^2 / 4 A^4 B + O( a^3 ).
    \end{equation}
    Taking $a \to 0$ and using that $L(0) = l$ gives that $L''(0) \leq - l |\zeta|^2 / 4 A^4 B$, so setting $C = 1/4 A^4 B$, we find we have proved what was required. \qedhere
\end{proof}

\section{Analysis of Regime I via Density Methods} \label{regime1firstsection}

\subsection{\boldmath $L^2$ Estimates For Bounded Density Inputs}

We now begin obtaining bounds for the operator $T^I$ specified in Proposition \ref{TjbLemma} by using the quasi-orthogonality estimates of Proposition \ref{theMainEstimatesForWave}. Define a metric $d_M = d_M^+ + d_M^-$ on $M$. Given an input $u: M \to \CC$, we consider a maximal $1/R$ separated subset $\mathcal{X}_R$ of $M$, and then consider a decomposition $u = \sum_{x_0 \in \mathcal{X}_R} u_{x_0}$ with respect to some partition of unity, where $\text{supp}(u_{x_0}) \subset B(x_0,1/R)$. The balls $\{ B(x_0,1/R) : x_0 \in \mathcal{X}_R \}$ have finite overlap, and so
\begin{equation}
    \| u \|_{L^p(M)} \sim \left( \sum\nolimits_{x_0 \in \mathcal{X}_R} \| u_{x_0} \|_{L^p(M)}^p \right)^{1/p}.
\end{equation}
If we set $f_{x_0,t_0} = T_{t_0}^I \{ u_{x_0} \}$, then
\begin{equation} \label{DAPOCJAPWOCJAWOIFJOI}
    \left\| T^I u \right\|_{L^p(M)} = \left\| \sum\nolimits_{(x_0,t_0) \in \mathcal{X}_R \times \mathcal{T}_R} f_{x_0,t_0} \right\|_{L^p(M)}.
\end{equation}
In this subsection, we use the quasi-orthogonality estimates of the last section to obtain $L^2$ estimates on partial sums of a family of functions of the form $\{ {S\!}_{x_0,t_0} \}$, which are essentially $L^1$ normalized versions of the functions $\{ f_{x_0, t_0} \}$, under a density assumption on the set of indices we are summing over. Namely, we say a set $\mathcal{E} \subset \mathcal{X}_R \times \mathcal{T}_R$ has \emph{density type} $(A_0,A_1)$ if for any set $B \subset \mathcal{X}_R \times \mathcal{T}_R$ with $1/R \leq \text{diam}(B) \leq A_1/R$,
\begin{equation}
    \#(\mathcal{E} \cap B) \leq R A_0\; \text{diam}(B). \footnote[1]{This definition of density is chosen because it is `scale-invariant' as we change the parameter $R$. Indeed, if $M = \RR^d$, $\mathcal{X}_R = (\ZZ / R)^d$, and $\mathcal{T}_R = \ZZ/R$, then a set $\mathcal{E} \subset (\ZZ / R)^d \times (\ZZ/R)$ has density type $(A_0,A_1)$ if and only if $R\; \mathcal{E} \subset \ZZ^d \times \ZZ$ has density type $(A_0,A_1)$.}
\end{equation}
To obtain $L^p$ bounds from these $L^2$ bounds, in the next section we will perform a \emph{density decomposition} to break up $\mathcal{X}_R \times \mathcal{T}_R$ into families of indices with controlled density, and then apply Proposition \ref{L2DensityProposition} on each subfamily to control \eqref{DAPOCJAPWOCJAWOIFJOI} via an interpolation.

\begin{prop} \label{L2DensityProposition}
    Fix $A \geq 1$. Consider a set $\mathcal{E} \subset \mathcal{X}_R \times \mathcal{T}_R$. Suppose that for each $(x_0,t_0) \in \mathcal{E}$, we pick two measurable functions $b_{t_0}: I_0 \to \RR$ and $u_{x_0}: M \to \RR$, supported on $I_{t_0}$ and $B(x_0,1/R)$ respectively, such that $\| b_{t_0} \|_{L^1(I_0)} \leq 1$ and $\| u_{x_0} \|_{L^1(M)} \leq 1$. Define 
    \begin{equation}
        {S\!}_{x_0,t_0} = \int b_{t_0}(t) (Q_R \circ e^{2 \pi i t P} \circ Q_R) u_{x_0}\; dt.
    \end{equation}
    Write $\mathcal{E} = \bigcup_{k = 0}^\infty \mathcal{E}_k$, where
    \begin{equation}
        \mathcal{E}_0 = \{ (x,t) \in \mathcal{E}: 0 \leq |t| \leq 1/R \}
    \end{equation}
    and for $k > 0$, define
    \begin{equation}
        \mathcal{E}_k = \{ (x,t) \in \mathcal{E}: 2^{k-1} / R < |t| \leq 2^k / R \}.
    \end{equation}
    Suppose that for each $k$, the set $\mathcal{E}_k$ has density type $(A,2^{k})$.
    %
    %
    Then
    \begin{equation} \label{DOIAWJDOIWAJDOIwdj21312321}
        \Big\| \sum\nolimits_k \sum\nolimits_{(x_0,t_0) \in \mathcal{E}_k} 2^{k \frac{d-1}{2}} {S\!}_{x_0,t_0} \Big\|_{L^2(M)}^2 \lesssim R^d \log(A) A^{\frac{2}{d-1}} \sum\nolimits_k 2^{k(d-1)} \# \mathcal{E}_k.
    \end{equation}
\end{prop}

\begin{remark}
    If $\| b_{t_0} \|_{L^1(I_0)} \sim 1$ and $\| u_{x_0} \|_{L^1(M)} \sim 1$, then locally constancy from the uncertainty principle and energy conservation of the wave equation tell us that morally,
    \begin{equation}
        \| S_{x_0,t_0} \|_{L^2(M)}^2 \sim R^d t_0^{d-1}.
    \end{equation}
    If $\| b_{t_0} \|_{L^1(I_0)} \sim 1$ and $\| u_{x_0} \|_{L^1(M)} \sim 1$ for all $(x_0,t_0) \in \mathcal{E}$, this means that Proposition \ref{L2DensityProposition} is morally equivalent to
    \begin{equation}
        \Big\| \sum\nolimits_{(x_0,t_0) \in \mathcal{E}} {S\!}_{x_0,t_0} \Big\|_{L^2(M)} \lesssim \sqrt{\log(A)} A^{\frac{1}{d-1}} \left( \sum\nolimits_{(x_0,t_0) \in \mathcal{E}} \| {S\!}_{x_0,t_0} \|_{L^2(M)}^2 \right)^{1/2}.
    \end{equation}
    Thus Proposition \ref{L2DensityProposition} is a kind of square root cancellation bound, albeit with an implicit constant which grows as the set $\mathcal{E}$ increases in density, a necessity given that the functions $\{ {S\!}_{x_0,t_0} \}$ are not almost-orthogonal to one another.
\end{remark}

\begin{proof}
Write $F = \sum_k F_k$, where
\begin{equation}
    F_k = 2^{k \frac{d-1}{2}} \sum\nolimits_{(x_0,t_0) \in \mathcal{E}_k} {S\!}_{x_0,t_0}.
\end{equation}
Our goal is to bound $\| F \|_{L^2(M)}$. Applying Cauchy-Schwarz, we have
\begin{equation} \label{loglossbound}
    \| F \|_{L^2(M)}^2 \leq \log(A) \left( \sum\nolimits_{k \leq \log(A)} \| F_k \|_{L^2(M)}^2 + \left\| \sum\nolimits_{k \geq \log(A)} F_k \right\|_{L^2(M)}^2 \right).
\end{equation}
Without loss of generality, increasing the implicit constant in the final result by applying the triangle inequality, we can assume that $\{ k : \mathcal{E}_k \neq \emptyset \}$ is $10$-separated, and that all values of $t$ with $(x,t) \in \mathcal{E}$ are positive. Thus if $F_k$ and $F_{k'}$ are both nonzero functions, then $k = k'$ or $|k - k'| \geq 10$.

%
%
Let us estimate $\langle F_k, F_{k'} \rangle$ for $k \geq k' + 10$.
We write
\begin{equation}
    \langle F_k, F_{k'} \rangle = \sum\nolimits_{(x_0,t_0) \in \mathcal{E}_k} \sum\nolimits_{(x_1,t_1) \in \mathcal{E}_{k'}} 2^{k \frac{d-1}{2}} 2^{k' \frac{d-1}{2}} \langle {S\!}_{x_0,t_0}, {S\!}_{x_1,t_1} \rangle.
\end{equation}
For each $(x_0,t_0) \in \mathcal{E}_k$, and each $k' \leq k - 10$, consider the set
\begin{equation}
    \mathcal{G}_0(x_0,t_0,k') = \{ (x_1,t_1) \in \mathcal{E}_{k'} : |(t_0 - t_1) - d_M^-(x_0,x_1)| \leq 2^{k' + 5} / R \},
\end{equation}
Also consider the sets of indices
\begin{equation}
    \mathcal{G}_l^+(x_0,t_0,k') = \{ (x_1,t_1) \in \mathcal{E}_{k'} : 2^{l} / R < |(t_0 - t_1) + d_M^+(x_0,x_1)| \leq 2^{l + 1} / R \}.
\end{equation}
and
\begin{equation}
    \mathcal{G}_l^-(x_0,t_0,k') = \{ (x_1,t_1) \in \mathcal{E}_{k'} : 2^{l} / R < |(t_0 - t_1) - d_M^-(x_0,x_1)| \leq 2^{l + 1} / R \}.
\end{equation}
If we set
\begin{equation}
    \mathcal{G}_0(x_0,t_0,k') = (\mathcal{G}_l^+(x_0,t_0,k') \cup \mathcal{G}_l^-(x_0,t_0,k'))
\end{equation}
and
\begin{equation}
\begin{split}
    \mathcal{G}_l(x_0,t_0,k') &= \Big( \mathcal{G}_l^+(x_0,t_0,k') \cup \mathcal{G}_l^-(x_0,t_0,k') \Big)\\
    &\quad\quad\quad\quad\quad\quad - \bigcup\nolimits_{r < l} \Big(\mathcal{G}_r^+(x_0,t_0,k') \cup \mathcal{G}_r^-(x_0,t_0,k') \Big).
\end{split}
\end{equation}
Then $\mathcal{E}_{k'}$ is covered by $\mathcal{G}_0(x_0,t_0,k')$ and $\mathcal{G}_l(x_0,t_0,k')$ for $k' + 5 \leq l \leq 10 \log R$. Define
\begin{equation}
    B_0(x_0,t_0,k') = \sum\nolimits_{(x_1,t_1) \in \mathcal{G}_0(x_0,t_0,k')} 2^{k \frac{d-1}{2}} 2^{k' \frac{d-1}{2}} |\langle {S\!}_{x_0,t_0}, {S\!}_{x_1,t_1} \rangle|,
\end{equation}
and
\begin{equation}
    B_l(x_0,t_0,k') = \sum\nolimits_{(x_1,t_1) \in \mathcal{G}_l(x_0,t_0,k')} 2^{k \frac{d-1}{2}} 2^{k' \frac{d-1}{2}} |\langle {S\!}_{x_0,t_0}, {S\!}_{x_1,t_1} \rangle|.
\end{equation}
We thus have
\begin{equation}
    \langle F_k, F_{k'} \rangle \leq \sum\nolimits_{(x_0,t_0) \in \mathcal{E}_k} B_0(x_0,t_0,k') + \sum\nolimits_{(x_0,t_0) \in \mathcal{E}_k} \sum\nolimits_{k' + 5 \leq l \leq 10 \log R} B_l(x_0,t_0,k').
\end{equation}
Using the density properties of $\mathcal{E}$, we can control the size of the index sets $\mathcal{G}_{\bullet}(x_0,t_0,k')$, and thus control the quantities $B_\bullet(x_0,t_0,k')$. The rapid decay of Proposition \ref{theMainEstimatesForWave} means that only $\mathcal{G}_0(x_0,t_0,k')$ needs to be estimated rather efficiently:
\begin{itemize}[leftmargin=8mm]
    \item Start by bounding the quantities $B_0(x_0,t_0,k')$. If $(x_1,t_1) \in \mathcal{G}_0(x_0,t_0,k')$, then
    \begin{equation}
        |d_M^-(x_0,x_1) - (t_0 - t_1)| \leq 2^{k'+5}/R,
    \end{equation}
    Thus if we consider $\text{Ann}(x_0,t_0,k') = \{ x_1: |d_M^-(x_0,x_1) - t_0| \leq 2^{k'+8}/R \}$, which is a geodesic annulus of radius $\sim 2^k / R$ and thickness $O(2^{k'}/R)$, then
    \begin{equation}
        \mathcal{G}_0(x_0,t_0,k') \subset \text{Ann}(x_0,t_0,k') \times [ 2^{k'}/R, 2^{k'+1}/R ].
    \end{equation}
    The latter set is covered by $O(2^{(k-k')(d-1)})$ balls of radius $2^{k'}/R$, and so the density properties of $\mathcal{E}_{k'}$ implies that
    \begin{equation} \label{G0Size}
        \#( \mathcal{G}_0(x_0,t_0,k') ) \lesssim A 2^{(k-k')(d-1)} 2^{k'}. 
    \end{equation}
    Since $k \geq k' + 10$, for $(x_1,t_1) \in \mathcal{G}_0(x_0,t_0,k')$ we have $d_M(x_0,x_1) \gtrsim 2^k / R$ and so
    \begin{equation} \label{InnerProductG0Size}
        \langle S_{x_0,t_0}, S_{x_1,t_1} \rangle \lesssim R^d 2^{-k \left( \frac{d-1}{2} \right)}.
    \end{equation}
    But putting together \eqref{G0Size} and \eqref{InnerProductG0Size} gives that
    \begin{align*}
        B_0(x_0,t_0,k') &\leq (2^{k \frac{d-1}{2}} 2^{k' \frac{d-1}{2}}) ( A 2^{(k-k')(d-1)} 2^{k'} )  ( R^d 2^{-k \left( \frac{d-1}{2} \right)} )\\
        &= A R^d 2^{k(d-1)} 2^{-k' \frac{d-3}{2}}.
    \end{align*}
    Thus for each $k$, since $d \geq 4$,
    \begin{equation} \label{AAAlowbounds}
        \sum\nolimits_{(x_0,t_0) \in \mathcal{E}_k} \sum\nolimits_{k' \in [\log(A), k - 10]} B_0(x_0,t_0,k') \lesssim R^{d} 2^{k (d-1)} \# \mathcal{E}_k.
    \end{equation}

    \item Next we bound $B_l(x_0,t_0,k')$ for $k' + 5 \leq l \leq k - 5$. The set $\mathcal{G}_l^+(x_0,t_0,k')$ is empty in this case. Thus
    \begin{equation}
        \mathcal{G}_l(x_0,t_0,k') \subset ( \text{Ann} \cup \text{Ann}' ) \times [t_0 - 2^{k'} / R, t_0 + 2^{k'} / R],
    \end{equation}
    where
    \begin{equation}
        \text{Ann} = \{ x \in M: |d_M^-(x_0,x) - (t_0 - 2^{k'}) / R| \leq 100 \cdot 2^l / R \}
    \end{equation}
    and
    \begin{equation}
        \text{Ann}' = \{ x \in M: |d_M^-(x_0,x) - (t_0 + 2^{k'}) / R| \leq 100 \cdot 2^l / R \},
    \end{equation}
    These are geodesic annuli of thickness $O(2^l / R)$ and radius $\sim 2^k$. Thus $\mathcal{G}_l(x_0,t_0,k')$ is covered by $O( 2^{(l-k')} 2^{(k-k')(d-1)} )$ balls of radius $2^{k'} / R$, and the density of $\mathcal{E}_{k'}$ implies that
    \begin{equation}
        \#(\mathcal{G}_l(x_0,t_0,k')) \lesssim R A\; 2^{(l-k')} 2^{(k-k')(d-1)} 2^{k'} / R = A 2^{l} 2^{(k-k')(d-1)}.
    \end{equation}
    For $(x_1,t_1) \in \mathcal{G}_l(x_0,t_0,k')$, $d_M(x_0,x_1) \sim 2^k / R$, and thus Proposition \ref{theMainEstimatesForWave} implies
    \begin{equation}
        |\langle {S\!}_{x_0,t_0}, {S\!}_{x_1,t_1} \rangle| \lesssim R^d 2^{-k \frac{d-1}{2}} 2^{-lK}.
    \end{equation}
    Thus for any $K \geq 0$,
    \begin{equation}
    \begin{split}
        B_l(x_0,t_0,k') &\lesssim_K \Big( A 2^{l} 2^{(k-k')(d-1)} \Big)  R^{d} 2^{k \frac{d-1}{2}} 2^{k' \frac{d-1}{2}} \Big( 2^{-k \frac{d-1}{2}} 2^{-lK} \Big)\\
        &\lesssim A R^d 2^l 2^{k(d-1)} 2^{-k' \frac{d-1}{2}} 2^{-lK}.
    \end{split}
    \end{equation}
    Picking $K > 1$, we conclude that
    \begin{equation} \label{AAAlBoundSmall}
        \sum_{(x_0,t_0) \in \mathcal{E}_k} \sum_{k' \in [\log(A), k - 10]} \sum_{l \in [k' + 10, k - 5]} B_l(x_0,t_0,k') \lesssim R^{d} 2^{k (d-1)} \# \mathcal{E}_k.
    \end{equation}

    \item Finally, let's bound $B_l(x_0,t_0,k')$ for $k - 5 \leq l \leq 10 \log R$. If either $(x_1,t_1) \in \mathcal{G}_l^-(x_0,t_0,k')$ or $(x_1,t_1) \in \mathcal{G}_l^+(x_0,t_0,k')$, then $d_M(x_0,x_1) \lesssim 2^l / R$. So $\mathcal{G}_l(x_0,t_0,k')$ is covered by $O( 2^{(l-k')d} )$ balls of radius $2^{k'} / R$, and thus
    \begin{equation}
        \#(\mathcal{G}_l(x_0,t_0,k')) \lesssim R A\; 2^{(l-k')d} (2^{k'} / R) = A 2^{(l-k')d} 2^{k'}.
    \end{equation}
    For $(x_1,t_1) \in \mathcal{G}_l(x_0,t_0,k')$, we have no good control over $d_M(x_0,t_1)$ aside from the trivial estimate $d_M(x_0,x_1) \lesssim 1$. Thus Proposition \ref{theMainEstimatesForWave} yields a bound of the form
    \begin{equation}
        |\langle {S\!}_{x_0,t_0}, {S\!}_{x_1,t_1} \rangle| \lesssim R^d 2^{-lK}.
    \end{equation}
    Thus we conclude that
    \begin{equation}
    \begin{split}
        B_l(x_0,t_0,k') &\lesssim_N R^{d} 2^{k \frac{d-1}{2}} 2^{k' \frac{d-1}{2}} \Big( A 2^{(l-k')d} 2^{k'} \Big) \Big( 2^{-lN} \Big)\\
        &= A R^d 2^{k \frac{d-1}{2}} 2^{-k' \frac{d-1}{2}} 2^{-lN}
    \end{split}
    \end{equation}
    Picking $K > d$, we conclude that
    \begin{equation} \label{AAAlBoundBig}
        \sum\nolimits_{(x_0,t_0) \in \mathcal{E}_k} \sum\nolimits_{k' \in [\log(A), k - 10]} \sum\nolimits_{l \in [k+10,\log R]} B_l(x_0,t_0,k')  \lesssim R^d.
    \end{equation}
\end{itemize}
The three bounds \eqref{AAAlowbounds}, \eqref{AAAlBoundSmall} and \eqref{AAAlBoundBig} imply that
\begin{equation} \label{DADAOWIDJAWOIDJAWDaweq13412}
    \sum\nolimits_k \sum\nolimits_{k' \in [\log(A), k]} |\langle F_k, F_{k'} \rangle| \lesssim R^d \sum\nolimits_k 2^{k (d-1)} \# \mathcal{E}_k.
\end{equation}
In particular, combining \eqref{DADAOWIDJAWOIDJAWDaweq13412} with \eqref{loglossbound}, we have
\begin{equation} \label{DPOWADPAWKDPOWAKDOPWAK}
    \| F \|_{L^2(M)}^2 \lesssim \log(A) \left( \sum\nolimits_k \| F_k \|_{L^2(M)}^2 + R^d \sum\nolimits_k 2^{k (d-1)} \# \mathcal{E}_k \right).
\end{equation}
Next, consider some parameter $a$ to be determined later, and decompose the interval $[2^{k} / R, 2^{k+1} / R]$ into the disjoint union of length $A^a / R$ intervals of the form
\begin{equation}
    I_{k,\mu} = [ 2^{k} / R + (\mu - 1) A^a / R, 2^{k} / R + \mu A^a / R] \quad\text{for $1 \leq \mu \leq 2^k/A^a$}.
\end{equation}
We thus consider a further decomposition $\mathcal{E}_k = \bigcup \mathcal{E}_{k,\mu}$, where $F_k = \sum F_{k,\mu}$. As before, increasing the implicit constant in the Proposition, we may assume without loss of generality that the set $\{ \mu: \mathcal{E}_{k,\mu} \neq \emptyset \}$ is $10$-separated. We now estimate
\begin{equation}
    \sum\nolimits_{\mu \geq \mu' + 10} |\langle F_{k,\mu}, F_{k,\mu'} \rangle|.
\end{equation}
For $(x_0,t_0) \in \mathcal{E}_{k,\mu}$ and $l \geq 1$, define
\begin{equation}
    \mathcal{H}_l(x_0,t_0,\mu') = \Big\{ (x_1,t_1) \in \mathcal{E}_{k,\mu'} : \frac{2^l A^a}{2R} \leq \max(d_M(x_0,x_1), t_0 - t_1) \leq \frac{2^l A^a}{R} \Big\}.
\end{equation}
Then $\bigcup_{l \geq 1} \mathcal{H}_l(x_0,t_0,\mu')$ covers $\bigcup_{\mu \geq \mu' + 10} \mathcal{E}_{k,\mu'}$. Set
\begin{equation}
    B'_l(x_0,t_0,\mu') = \sum\nolimits_{(x_1,t_1) \in \mathcal{H}_l(x_0,t_0,\mu')} 2^{k(d-1)} |\langle {S\!}_{x_0,t_0}, {S\!}_{x_1,t_1} \rangle|.
\end{equation}
Then
\begin{equation}
    \langle F_{k,\mu}, F_{k,\mu'} \rangle \leq \sum\nolimits_{(x_0,t_0) \in \mathcal{E}_{k,\mu}} \sum\nolimits_l B'_l(x_0,t_0,\mu').
\end{equation}
We now bound the constants $B'_l$. Pick a constant $r$ such that $d_M \leq 2^r d_M^+$ and $d_M \leq 2^r d_M^-$. As in the estimates of the quantities $B_l$, the quantities where $l$ is large have negligible magnitude:
\begin{itemize}[leftmargin=8mm]
    \item For $l \leq k - a \log_2 A + 10 r$, we have $2^l A^a / R \lesssim 2^k / R$. The set $\mathcal{H}_l(x_0,t_0,\mu')$ is covered by $O(1)$ balls of radius $2^l A^a / R$, and density properties imply
    \begin{equation}
        \# \mathcal{H}_l(x_0,t_0,\mu') \lesssim (R A) (2^l A^a / R) = A^{a+1} 2^l
    \end{equation}
    For $(x_1,t_1) \in \mathcal{H}_l(x_0,t_0,\mu')$, we claim that
    \begin{equation}
        2^{k(d-1)} |\langle {S\!}_{x_0,t_0}, {S\!}_{x_1,t_1} \rangle| \lesssim R^{d} 2^{k(d-1)} (2^l A^a)^{- \frac{d-1}{2}}.
    \end{equation}
    Indeed, for such tuples we have
    \begin{equation}
        d_M(x_0,x_1) \gtrsim 2^l A^a / R \quad\text{or}\quad \min\nolimits_{\pm} |d_M^{\pm}(x_0,x_1) - (t_0 - t_1)| \gtrsim 2^l A^a / R,
    \end{equation}
    and the estimate follows from Proposition \ref{theMainEstimatesForWave} in either case. Since $d \geq 4$, we conclude that
    \begin{align} \label{BBBEquation}
    \begin{split}
        &\sum\nolimits_{l \in [1, k - a \log_2 A + 10]} B'_l(x_0,t_0,,\mu')\\
        &\quad\quad\quad\quad \lesssim \sum\nolimits_{l \in [1, k - a \log_2 A + 10]} R^{d} (2^{k(d-1)}) (2^l A^a)^{- \frac{d-1}{2}} (A^{a+1} 2^l)\\
        &\quad\quad\quad\quad \lesssim \sum\nolimits_{l \in [1, k - a \log_2 A + 10]} R^{d}  2^{k(d-1)} 2^{-l \frac{d-3}{2}} A^{1 - a \left( \frac{d-3}{2} \right)}\\
        &\quad\quad\quad\quad \lesssim R^{d} 2^{k(d-1)} A^{1 - a \left( \frac{d-3}{2} \right)}.
    \end{split}
    \end{align}

    \item For $l > k - a \log_2 A + 10 r$, a tuple $(x_1,t_1) \in \mathcal{E}_k$ lies in $\mathcal{H}_l(x_0,t_0,\mu')$ if and only if $2^l A^a / 2 R \leq d_M(x_0,x_1) \leq 2^l A^a / R$, since we always have
    \begin{equation}
         t_0 - t_1 \leq 2^{k+1}/R < 2^{l+r} A^a / 8R.
    \end{equation}
    and so $d_M(x_0,x_1) \geq 2^l A^a / 2R$. And so 
    \[ |(t_0 - t_1) - d_M^-(x_0,x_1)| \geq 2^{-r} 2^l A^a / 2R - 2^{k+1} / R \geq 2^{-r} 2^l A^a / 4R. \]
    Also $|(t_0 - t_1) + d_M^+(x_0,x_1)| \geq |d_M^+(x_0,x_1) \geq 2^{-r} 2^l A^a / 4R$. Thus we conclude from Proposition \ref{theMainEstimatesForWave} that
    \begin{equation}
        2^{k(d-1)} |\langle {S\!}_{x_0,t_0}, {S\!}_{x_1,t_1} \rangle| \lesssim_K R^{d} 2^{k(d-1)} (2^l A^a)^{- K}.
    \end{equation}
    Now $\mathcal{H}_l(x_0,t_0,\mu')$ is covered by $O( (2^{l-k} A^a)^d )$ balls of radius $2^k / R$, and the density properties of $\mathcal{E}_k$ thus imply that
    \begin{equation}
        \#(\mathcal{H}_l(x_0,t_0,\mu')) \lesssim (RA) (2^{l-k} A^a)^d ( 2^k / R ) \lesssim A^{1 + ad} 2^{ld} 2^{-k(d-1)}.
    \end{equation}
    Thus, picking $K > \max(d,1+ad)$, we conclude that
    \begin{align} \label{BBB2}
    \begin{split}
        &\sum\nolimits_{l \geq k - a \log_2 A + 10} B'_l(x_0,t_0,\mu')\\
        &\quad \lesssim R^{d} \sum\nolimits_{l \geq k - a \log_2 A + 10} (2^{k(d-1)}) (2^l A^a)^{-M} A^{1 + ad} 2^{ld} 2^{-k(d-1)} \lesssim R^{d}.
    \end{split}
    \end{align}
    \end{itemize}
    Combining \eqref{BBBEquation} and \eqref{BBB2}, and then summing over the tuples $(x_0,t_0) \in \mathcal{E}_{k,\mu}$, we conclude that
    \begin{equation} \label{DOUIAWJDOIAWJVIO}
        \sum\nolimits_{\mu \geq \mu' + 10} |\langle F_{k,\mu}, F_{k,\mu'} \rangle| \lesssim R^{d} \left( 1 + 2^{k(d-1)} A^{1 - a \left( \frac{d-3}{2} \right)} \right) \# \mathcal{E}_{k,\mu}.
    \end{equation}
    Now summing in $\mu$, \eqref{DOUIAWJDOIAWJVIO} implies that
    \begin{equation} \label{DAOWDHAODWWID}
        \| F_k \|_{L^2(M)}^2 \lesssim \sum\nolimits_\mu \| F_{k,\mu} \|_{L^2(M)}^2 + R^{d} \left( 1 + 2^{k(d-1)} A^{1 - a \left( \frac{d-3}{2} \right)} \right) \# \mathcal{E}_k.
    \end{equation}
The functions in the sum defining $F_{k,\mu}$ are highly coupled, and it is difficult to use anything except Cauchy-Schwarz to break them apart. Since $\# ( \mathcal{T}_R \cap I_{k,\mu}) \sim A^a$, if we set $F_{k,\mu} = \sum_{t \in \mathcal{T}_R \cap I_{k,\mu}} F_{k,\mu,t}$, where
\begin{equation}
    F_{k,\mu,t} = \sum\nolimits_{(x_0,t) \in \mathcal{E}_{k,\mu}} 2^{k \frac{d-1}{2}} {S\!}_{x_0,t}.
\end{equation}
Then Cauchy-Schwarz implies that
\begin{equation} \label{IOJDAOIWDJAWOIJF}
    \| F_{k,\mu} \|_{L^2(M)}^2 \lesssim A^a \sum\nolimits_{t \in \mathcal{T}_R \cap I_{k,\mu}} \| F_{k,\mu,t} \|_{L^2(M)}^2.
\end{equation}
Since the elements of $\mathcal{X}_R$ are $1/R$ separated, the functions in the sum defining $F_{k,\mu,t}$ are quite orthogonal to one another; Proposition \ref{theMainEstimatesForWave} implies that for $x_0 \neq x_1$,
\begin{equation}
    |\langle {S\!}_{x_0,t}, {S\!}_{x_1,t} \rangle| \lesssim R^d ( R d_M(x_0,x_1) )^{-K} \quad\text{for all $K \geq 0$}.
\end{equation}
Thus
\begin{equation}
    \| F_{k,\mu,t} \|_{L^2(M)}^2 \lesssim R^{d} 2^{k(d-1)} \# (\mathcal{E}_k \cap (M \times \{ t \})).
\end{equation}
But this means that
\begin{equation} \label{eoqiejoiwjdoiaevjoa}
    A^a \sum\nolimits_{t \in \mathcal{T}_R \cap I_{k,\mu}} \| F_{k,\mu,t} \|_{L^2(M)}^2 \lesssim R^{d} 2^{k(d-1)} A^a \# \mathcal{E}_{k,\mu}.
\end{equation}
Thus \eqref{DAOWDHAODWWID}, \eqref{IOJDAOIWDJAWOIJF}, and \eqref{eoqiejoiwjdoiaevjoa} imply that
\begin{equation}
\begin{split}
    \| F_k \|_{L^2(M)}^2 &\lesssim \sum\nolimits_\mu \| F_{k,\mu} \|_{L^2(M)}^2 + R^{d} \left( 1 + 2^{k(d-1)} A^{1 - a \left( \frac{d-3}{2} \right)} \right) \# \mathcal{E}_k\\
    &\lesssim R^{d} \left( 2^{k(d-1)} A^a + (1 + 2^{k(d-1)} A^{1 - a \left( \frac{d-3}{2} \right)} \right) \# \mathcal{E}_k.
\end{split}
\end{equation}
%
Optimizing by picking $a = 2 / (d-1)$ gives that
\begin{equation} \label{OICJOAIEVJAIOJFAOIJRIO}
    \| F_k \|_{L^2(M)}^2 \lesssim R^{d} 2^{k(d-1)} A^{\frac{2}{d-1}} \# \mathcal{E}_k.
\end{equation}
The proof is completed by combining \eqref{DPOWADPAWKDPOWAKDOPWAK} with \eqref{OICJOAIEVJAIOJFAOIJRIO}.
\end{proof}

\subsection{\boldmath $L^p$ Estimates Via Density Decompositions} \label{regime1densitydecomposition}

Combining the $L^2$ analysis of Section \ref{regime1firstsection} with a density decomposition argument, we can now prove the following Lemma, which completes the analysis of the operator $T^I$ in Proposition \ref{TjbLemma}.

\begin{lemma} \label{regime1Lemma}
    Using the notation of Proposition \ref{TjbLemma}, let $T^I = \sum\nolimits_{t_0 \in \mathcal{T}_R} T^I_{t_0}$, where
    \begin{equation}
        T^I_{t_0} = b_{t_0}^I(t) (Q_R \circ e^{2 \pi i t P} \circ Q_R)\; dt.
    \end{equation}
    Then for $1 \leq p < 2 (d-1) / (d+1)$,
    \begin{equation}
        \| T^I u \|_{L^p(M)} \lesssim R^{-1/p'} \left( \sum\nolimits_{t_0 \in \mathcal{T}_R} \Big[ \| b^I_{t_0} \|_{L^p(I_0)} \langle R t_0 \rangle^{\alpha(p)} \Big]^p \right)^{1/p} \| u \|_{L^p(M)}.
    \end{equation}
\end{lemma}

We prove Lemma \ref{regime1Lemma} via a \emph{density decomposition} argument, adapted from the methods of \cite{HeoandNazarovandSeeger}. Given a function $u: M \to \CC$, we use a partition of unity to write
\begin{equation}
    u = \sum\nolimits_{x_0 \in \mathcal{X}_R} u_{x_0},
\end{equation}
where $u_{x_0}$ is supported on $B(x_0,1/R)$, and
\begin{equation}
\begin{split}
    \left( \sum\nolimits_{x_0 \in \mathcal{X}_R} \| u_{x_0} \|_{L^1(M)}^p \right)^{1/p} &\lesssim R^{-d/p'} \left( \sum\nolimits_{x_0 \in \mathcal{X}_R} \| u_{x_0} \|_{L^p(M)}^p \right)^{1/p} \lesssim R^{-d/p'} \| u \|_{L^p(M)}.
\end{split}
\end{equation}
Define
\begin{equation}
    \mathcal{X}_{a} = \{ x_0 \in \mathcal{X}_R: 2^{a-1} < \| u_{x_0} \|_{L^1(M)} \leq 2^a \}
\end{equation}
and let
\begin{equation}
    \mathcal{T}_{b} = \{ t_0 \in \mathcal{T}_R: 2^{b-1} < \| b_{t_0}^I \|_{L^1(M)} \leq 2^b \}.
\end{equation}
Define functions $f_{x_0,t_0} = T_{t_0}^I u_{x_0}$. Lemma \ref{regime1Lemma} follows from the following result.

\begin{lemma} \label{LpBoundLemma}
    Fix $u \in L^p(M)$, and consider $\mathcal{X}_{a}$, $\mathcal{T}_{b}$, and $\{ f_{x_0,t_0} \}$ as above. For any function $c: \mathcal{X}_R \times \mathcal{T}_R \to \CC$, and $1 < p < 2 (d-1) / (d+1)$,
    \begin{equation}
    \begin{split}
    &\Bigg\| \sum\nolimits_{a,b} \sum\nolimits_{(x_0,t_0) \in \mathcal{X}_{a} \times \mathcal{T}_{b}} 2^{-(a+b)} \langle R t_0 \rangle^{\frac{d-1}{2}} c(x_0,t_0) f_{x_0,t_0} \Big\|_{L^p(M)}\\
    &\quad\quad\quad\quad \lesssim R^{ d / p'} \left( \sum\nolimits_{(x_0,t_0) \in \mathcal{X}_{a} \times \mathcal{T}_{b}} |c(x_0,t_0)|^p \langle R t_0 \rangle^{d-1} \right)^{1/p}.
    \end{split}
    \end{equation}
\end{lemma}

To see how Lemma \ref{LpBoundLemma} implies Lemma \ref{regime1Lemma}, set $c(x_0,t_0) = 2^{a+b} \langle R t_0 \rangle^{- \frac{d-1}{2}}$ for $x_0 \in \mathcal{X}_{a}$ and $t_0 \in \mathcal{T}_{b}$. Then Lemma \ref{LpBoundLemma} implies that
\begin{equation}
\begin{split}
    &\| T^I u \|_{L^p(M)}\\
    &\quad = \left\| \sum f_{x_0,t_0} \right\|_{L^p(M)}\\
    &\quad \lesssim R^{d/p'} \left( \sum\nolimits_{(x_0,t_0)} \left[ \| b_{t_0}^I \|_{L^1(\RR)} \| u_{x_0} \|_{L^1(M)} \langle R t_0 \rangle^{\alpha(p)} \right]^p \right)^{1/p}\\
    &\quad \lesssim R^{-1/p'} \left( \sum\nolimits_{t_0} \Big[ \| b_{t_0}^I \|_{L^p(I_0)} \langle R t_0 \rangle^{\alpha(p)} \Big]^p \right)^{1/p}\left( \sum\nolimits_{x_0} \left[ \| u_{x_0} \|_{L^1(M)}  R^{d/p'} \right]^p \right)^{1/p}\\
    &\quad \lesssim R^{-1/p'} \left( \sum\nolimits_{t_0} \Big[ \| b_{t_0}^I \|_{L^p(I_0)} \langle R t_0 \rangle^{\alpha(p)} \Big]^p \right)^{1/p} \| u \|_{L^p(M)}.
\end{split}
\end{equation}
%
Thus we have proved Lemma \ref{regime1Lemma}. We take the remainder of this section to prove Lemma \ref{LpBoundLemma} using a density decomposition argument.

\begin{proof}[Proof of Lemma \ref{LpBoundLemma}]

For $p = 1$, this inequality follows simply by applying the triangle inequality, and applying the pointwise estimates of Proposition \ref{theMainEstimatesForWave}. By methods of interpolation, to prove the result for $p > 1$, we thus only need only prove a restricted strong type version of this inequality. In other words, we can restrict $c$ to be the indicator function of a set $\mathcal{E} \subset \mathcal{X}_R \times \mathcal{T}_R$. Write $\mathcal{E} = \bigcup_{k \geq 0} \mathcal{E}_{k,a,b}$, where
\begin{equation}
    \mathcal{E}_{0,a,b} = \{ (x,t) \in \mathcal{E} \cap (\mathcal{X}_{a} \times \mathcal{T}_{b}) : |t| \leq 1/R \}
\end{equation}
and for $k > 0$, let
\begin{equation}
    \mathcal{E}_{k,a,b} = \{ (x,t) \in \mathcal{E} \cap (\mathcal{X}_{a} \times \mathcal{T}_{b}) : 2^{k-1} / R < |t| \leq 2^{k} / R \}.
\end{equation}
Write
\begin{equation}
    F_k = \sum\nolimits_{a,b} \sum\nolimits_{(x_0,t_0) \in \mathcal{E}_{k,a,b}} 2^{k \left( \frac{d-1}{2} \right)} 2^{-(a+b)} f_{x_0,t_0}.
\end{equation}
Our proof will be completed if we can show that
\begin{equation} \label{oDOIAWJCVOIEJOIJER1312s}
    \Big\| \sum\nolimits_k F_k \Big\|_{L^p(M)} \lesssim R^{d ( 1 - 1/p )} \Big( \sum\nolimits_k 2^{k(d-1)} \# \mathcal{E}_k \Big)^{1/p}.
\end{equation}
To prove \eqref{oDOIAWJCVOIEJOIJER1312s}, we perform a density decomposition on the sets $\{ \mathcal{E}_k \}$. For $u \geq 0$, let $\widehat{\mathcal{E}}_k(u)$ be the set of all points $(x_0,t_0) \in \mathcal{E}_k$ that are contained in a ball $B$ with $\text{rad}(B) \leq 2^{k} / 100 R$, such that $\#( \mathcal{E}_k \cap B ) \geq R 2^{u} \text{rad}(B)$. Then define
\begin{equation}
    \mathcal{E}_k(u) = \widehat{\mathcal{E}}_k(u) - \bigcup\nolimits_{u' > u} \widehat{\mathcal{E}}_k(u').
\end{equation}
Because the set $\mathcal{E}_k$ is $1/R$ discretized, we have
\begin{equation}
    \mathcal{E}_k = \bigcup\nolimits_{u \geq 0} \mathcal{E}_k(u).
\end{equation}
Moreover $\mathcal{E}_k(u)$ has density type $(R 2^{u}, 2^{k} / 100 R)$, and thus by a covering argument, also has density type $(C_d R 2^{u}, 2^{k} / R)$ for $C_d = 1000^d$. Furthermore, there are disjoint balls $B_{k,u,1},\dots,B_{k,u,N_{k,u}}$ of radius at most $2^{k} / 100 R$ such that
\begin{equation}
    \sum\nolimits_n \text{rad}(B_{k,u,n}) \leq 2^{-u} / R \# \mathcal{E}_k.
\end{equation}
and such that $\mathcal{E}_k(u)$ is covered by the balls $\{ B_{k,u,n}^* \}$, where, for a ball $B$, $B^*$ denotes the ball with the same center as $B$, but 5 times the radius. Now write
\begin{equation}
    F_{k,u} = \sum\nolimits_{a,b} \sum\nolimits_{(x_0,t_0) \in \mathcal{E}_{k,a,b}(u)} 2^{k \left( \frac{d-1}{2} \right)} 2^{-(a+b)} f_{x_0,t_0}.
\end{equation}
Using the density assumption on $\mathcal{E}_k(u)$, we can apply Lemma \ref{L2DensityProposition} of the last section, which implies that, with ${S\!}_{x_0,t_0} = 2^{-(a+b)} f_{x_0,t_0}$ for $(x_0,t_0) \in \mathcal{E}_{k,a,b}(u)$,
\begin{equation} \label{DOIWAJOIAJVOIWAJFOIWF}
\begin{split}
    \Big\| \sum\nolimits_k F_{k,u} \Big\|_{L^2(M)} \lesssim R^{d/2} \left( u^{1/2} 2^{u \left( \frac{1}{d-1} \right)} \right) \left( \sum\nolimits_k 2^{k(d-1)} \# \mathcal{E}_k \right)^{1/2}.
\end{split}
\end{equation}
Let $(y_{k,u,n}, t_{k,u,n})$ denote the center of $B_{k,u,n}$. Then
\begin{equation}
    \sum\nolimits_{(x_0,t_0) \in [B_{k,u,n} \cap \mathcal{E}_k(u)]} 2^{-(a+b)} f_{x_0,t_0}
\end{equation}
has mass concentrated on the geodesic annulus $\text{Ann}_{k,u,n} \subset M$ with center $y_{k,u,n}$, with radius $t_{k,u,n} \sim 2^{k} / R$, and with thickness $5\; \text{rad}(B_{k,u,n})$. Thus
\begin{equation}
    \sum\nolimits_n |\text{Ann}_{k,u,n}| \lesssim \sum\nolimits_{n} (2^{k} / R)^{d-1} \text{rad}(B_{k,u,n}) \leq (2^{k}/R)^{d-1} R^{-1} 2^{-u} \# \mathcal{E}_k.
\end{equation}
If we set $\Lambda_u = \bigcup_k \bigcup_n \text{Ann}_{k,u,n}$, then
\begin{equation}
    |\Lambda_u| \lesssim R^{-d} 2^{-u} \sum\nolimits_k 2^{k(d-1)} \# \mathcal{E}_k
\end{equation}
Since $1/p - 1/2 > 1/(d-1)$, so that $\alpha(p) > 1$, H\"{o}lder's inequality implies that
\begin{equation}
\begin{split}
    \Big\| \sum\nolimits_k F_{k,u} \Big\|_{L^p(\Lambda_u)} &\lesssim |\Lambda_u|^{1/p - 1/2} \Big\| \sum\nolimits_k F_{k,u} \Big\|_{L^2(\Lambda_{k,u})}\\
    &\lesssim \left( R^{-d} 2^{-u} \sum\nolimits_k 2^{k(d-1)} \# \mathcal{E}_k \right)^{1/p - 1/2}\\
    &\quad\quad\quad\quad \left( R^{d} \left( u 2^{u \left( \frac{2}{d-1} \right)} \right) \sum\nolimits_k 2^{k(d-1)} \# \mathcal{E}_k \right)^{1/2}\\
    &= R^{d(1-1/p)} \left( u^{1/2} 2^{-u \left( \frac{\alpha(p) - 1}{d - 1} \right)}  \right) \left( \sum\nolimits_k 2^{k(d-1)} \# \mathcal{E}_k \right)^{1/p}\\
    &\lesssim R^{d(1 - 1/p)} 2^{-u \varepsilon} \left( \sum\nolimits_k 2^{k(d-1)} \# \mathcal{E}_k \right)^{1/p}
\end{split}
\end{equation}
for some suitable small $\varepsilon > 0$. For each $(x_0,t_0) \in \mathcal{E}_{k,a,b}(u) \cap B_{k,u,n}$, we calculate using the pointwise bounds for the functions $\{ f_{x_0,t_0} \}$ that
%
\begin{equation}
\begin{split}
        \| 2^{-(a+b)} f_{x_0,t_0} \|_{L^1(\Lambda(u)^c)} &= R^{d} \int_{\text{Ann}_R^c} \langle R d_M(x,x_0) \rangle^{- \left( \frac{d-1}{2} \right)} \langle R |t_0 - d_M(x,x_0)| \rangle^{-M}\; dx\\
        &\lesssim R^{ \left( \frac{d+1}{2} - M \right)} \int_{5\; \text{rad}(B_{k,u,n})}^{O(1)} ( t_{k,u,n} + s)^{\left( \frac{d-1}{2} \right)} s^{-M}\; ds\\
        &\lesssim R^{ \left( \frac{d+1}{2} - M \right)} \text{rad}(B_{k,u,n})^{1-M} t_{k,u,n}^{\left( \frac{d-1}{2} \right)}\\
        &\lesssim 2^{k \left( \frac{d-1}{2} \right)} (R \text{rad}(B_{k,u,n}))^{1 - M}.
\end{split}
\end{equation}
Thus
%
\begin{equation}
    \| 2^{k \left( \frac{d-1}{2} \right)} 2^{-(a+b)} f_{x_0,t_0} \|_{L^1(\text{Ann}_n^c)} \lesssim 2^{k(d-1)} (R \text{rad}(B_{k,u,n}))^{1 - M}
\end{equation}
%
%
%
%
Because the set of points in $\mathcal{E}_k$ is $1/R$ separated, there are at most $O( (R \text{rad}(B_{k,u,n}))^{d+1} )$ points in $\mathcal{E}_k(u) \cap B_{k,u,n}$, and so the triangle inequality implies that
\begin{equation}
\begin{split}
    &\Big\| \sum\nolimits_{a,b} \sum\nolimits_{(x_0,t_0) \in \mathcal{E}_{k,a,b}(u) \cap B_{k,u,n}} 2^{k \left( \frac{d-1}{2} \right)} 2^{-(a+b)} f_{x_0,t_0} \Big\|_{L^1(\Lambda(u)^c)}\\
    &\quad\quad\quad \lesssim 2^{k (d - 1)} (R \text{rad}(B_{k,u,n}))^{d + 2 - M}.
\end{split}
\end{equation}
Since $\# \mathcal{E}_k \cap B_{k,u,n} \geq R 2^{u}\ \text{rad}(B_{k,u,n})$, and $\mathcal{E}_k$ is $1/R$ discretized, we must have
\begin{equation}
    \text{rad}(B_{k,u,n}) \geq (2^u / 2^d)^{\frac{1}{d-1}}\; 1/R.
\end{equation}
Thus
%
%
\begin{equation}
\begin{split}
    \Big\| \sum\nolimits_k F_{k,u} \Big\|_{L^1(\Lambda(u)^c)} &\lesssim_M \sum\nolimits_k \sum\nolimits_n 2^{k (d-1)} (R \text{rad}(B_{k,u,n}))^{d + 2 - M}\\
    &\lesssim \sum\nolimits_k 2^{k(d-1)} \Big( R \min\nolimits_n \text{rad}(B_{k,u,n}) \Big)^{d + 1 - M}\\
    &\quad\quad\quad\quad \left( \sum\nolimits_n R \text{rad}(B_{k,u,n}) \right) \\
    &\lesssim \sum\nolimits_k 2^{k (d-1)} 2^{u \left( \frac{d+1-M}{d-1} \right)} \left( 2^{-u} \# \mathcal{E}_k \right)\\
    &\lesssim 2^{u \left( \frac{2-M}{d-1} \right)} \sum\nolimits_k 2^{k (d-1)} \# \mathcal{E}_k
\end{split}
\end{equation}
Picking $M > 2 + (1 - 1/p)(1/p - 1/2)^{-1}$, and interpolating with the bounds on $\| \sum_k F_{k,u} \|_{L^2(M)}$ yields that
\begin{equation}
\begin{split}
    \Big\| \sum\nolimits_k F_{k,u} \Big\|_{L^p(\Lambda(u)^c)} &\lesssim \left( 2^{u \left( \frac{2-M}{d-1} \right)} \right)^{2/p - 1} \left( R^{d/2} \left( u^{1/2} 2^{u \left( \frac{1}{d-1} \right)} \right) \right)^{2(1 - 1/p)}\\
    &\quad\quad\quad\quad \left( \sum 2^{k(d-1)} \# \mathcal{E}_k \right)^{1/p} \\
    &\lesssim 2^{-u \varepsilon} R^{d(1 - 1/p)} \sum\nolimits_k \left( \sum 2^{k(d-1)} \# \mathcal{E}_k \right)^{1/p}.
\end{split}
\end{equation}
So now we know
\begin{equation}
    \Big\| \sum\nolimits_k F_{k,u} \Big\|_{L^p(M)} \lesssim R^{d(1-1/p)} 2^{- u \varepsilon} \left( \sum\nolimits_k 2^{k(d-1)} \# \mathcal{E}_k \right)^{1/p}.
\end{equation}
The exponential decay in $u$ allows us to sum in $u$ to obtain that
\begin{equation}
    \Big\| \sum\nolimits_u \sum\nolimits_k F_{k,u} \Big\|_{L^p(M)} \lesssim R^{d(1 - 1/p)} \left( \sum\nolimits_k 2^{k(d-1)} \# \mathcal{E}_k \right)^{1/p}.
\end{equation}
This is precisely the bound we were required to prove.
\end{proof}

\section{Analysis of Regime II via Local Smoothing} \label{regime2finalsection}

In this section, we bound the operators $\{ T^{II} \}$, by a reduction to an endpoint local smoothing inequality, namely, the inequality that
\begin{equation} \label{thelocalsmoothinginequality}
    \| e^{2\pi i t P} f \|_{L^{p'}(M) L^{p'}_t(I_0)} \lesssim \| f \|_{L^{p'}_{\alpha(p) - 1/p'}}.
\end{equation}
This inequality is proved in Corollary 1.2 of \cite{LeeSeeger} for $1 < p < 2(d-1)/(d+1)$ for classical elliptic pseudodifferential operators $P$ satisfying the cosphere assumption of Theorem \ref{CpVersionOfTheorem}. The range of $p$ here is also precisely the range of $p$ in Theorem \ref{CpVersionOfTheorem}. Alternatively, Lemma \ref{LpBoundLemma} can be used to prove \eqref{thelocalsmoothinginequality} independently of \cite{LeeSeeger} in the same range by a generalization of the method of Section 10 of \cite{HeoandNazarovandSeeger}.

\begin{lemma} \label{LocalSmoothingLargeTimesTheorem}
    Using the notation of Proposition \ref{TjbLemma}, let
    \begin{equation}
        T^{II} = \int b^{II}(t) (Q_R \circ e^{2 \pi i t P} \circ Q_R)\; dt.
    \end{equation}
    For $1 < p < 2 (d-1)/(d+1)$, we then have
    \begin{equation}
        \| T^{II} u \|_{L^p(M)} \lesssim R^{\alpha(p) - 1/p'} \| b^{II} \|_{L^p(I_0)} \| u \|_{L^p(M)}.
    \end{equation}
\end{lemma}
\begin{proof}
    For each $R$, the \emph{class} of operators of the form $\{ T^{II} \}$ formed from a given function $b^{II}$ is closed under taking adjoints. Indeed, if $T^{II}$ is obtained from $b^{II}$, then $(T^{II})^*$ is obtained from the multiplier $\overline{b^{II}}$. Because of this self-adjointness, if we can prove that
    \begin{equation}
        \| T^{II} u \|_{L^{p'}(M)} \lesssim R^{\alpha(p) - 1/p'} \| b^{II} \|_{L^p(I_0)} \| u \|_{L^{p'}(M)},
    \end{equation}
    then we obtain the required result by duality. We apply this duality because it is easier to exploit local smoothing inequalities in $L^{p'}(M)$ since now $p' > 2$.

    We begin by noting that the operators $\{ Q_R \}$, being a bounded family of order zero pseudo-differential operators, are uniformly bounded on $L^{p'}(M)$. Thus
    \begin{equation}
    \begin{split}
        \| T^{II} u \|_{L^{p'}(M)} &= \Big\| Q_R \circ \Big( \int_{I_0} b^{II}(t) e^{2 \pi i tP} (Q_R u) \Big) \Big\|_{L^{p'}(M)}\\
        &\lesssim \Big\| \Big( \int_{I_0} b^{II}(t) e^{2 \pi i tP} (Q_R u) \Big) \Big\|_{L^{p'}(M)}.
    \end{split}
    \end{equation}
    Applying H\"{o}lder and Minkowski's inequalities, we find that
    \begin{equation}
    \begin{split}
        \| T^{II}u \|_{L^{p'}(M)} &\leq \| b^{II} \|_{L^p(\RR)} \Big\| \Big( \int_{I_0} |e^{2 \pi i t P} (Q_R u)|^{p'} \Big)^{1/p'} \Big\|_{L^{p'}(M)}.
    \end{split}
    \end{equation}
    Applying the endpoint local smoothing inequality \eqref{thelocalsmoothinginequality}, we conclude that
    \begin{equation}
    \begin{split}
        \| T^{II} u \|_{L^{p'}(M)} &\lesssim \| b^{II} \|_{L^p(\RR)}  \| e^{2 \pi i P} (Q_R u) \|_{L^{p'}_t L^{p'}_x}\\
        &\lesssim  \| b^{II} \|_{L^p(\RR)}  \| Q_R u \|_{L^q_{\alpha(p) - 1/p'}(M)},
    \end{split}
    \end{equation}
    Bernstein's inequality for compact manifolds (see \cite{Sogge}, Section 3.3) gives
    \begin{equation}
        \| Q_R u \|_{L^{q}_{\alpha(p) - 1/p'}(M)} \lesssim R^{\alpha(p) - 1/p'} \| u \|_{L^p(M)}.
    \end{equation}
    Thus we conclude that
    \begin{equation}
        \| T^{II}u \|_{L^{p'}(M)} \lesssim R^{\alpha(p) - 1/p'} \| b^{II} \|_{L^p(I_0)} \| u \|_{L^{p'}(M)},
    \end{equation}
    which completes the proof.
\end{proof}

Combining Lemma \ref{regime1Lemma} and Lemma \ref{LocalSmoothingLargeTimesTheorem} completes the proof of Proposition \ref{TjbLemma}, and thus of inequality \eqref{dyadicMainReulst}. Since \eqref{TrivialLowFrequencyBound} was already proven as a consequence of Lemma \ref{lowjLemma}, this completes the proof of Theorem \ref{CpVersionOfTheorem}, and thus the main results of the paper.

\section{Appendix}

In this appendix, we provide proofs of Lemmas \ref{decompositionLemma} and \ref{pseudodifferentialCoordinateLemma}.

\begin{proof} [Proof of Lemma \ref{decompositionLemma}]
    The intervals $\{ I_{t_0} : t_0 \in \mathcal{T}_R \}$ cover $[-\varepsilon,\varepsilon]$, and so we may consider an associated partition $\mathbb{I}_{[-\varepsilon,\varepsilon]} = \sum_{t_0} \chi_{t_0}$ where $\text{supp}(\chi_{t_0}) \subset I_{t_0}$ and $|\chi_{t_0}| \leq 1$. Define $b_{t_0}^I = \chi_{t_0} b_j$ and $b^{II} = (1 - \mathbb{I}_{[-\varepsilon,\varepsilon]} ) b$. Then $b = \sum_{t_0} b_{t_0}^I + b^{II}$, and the support assumptions are satisfied. It remains to prove the required norm bounds for these choices. For each $n \in \ZZ$, define a function $b_{n}: I_0 \to \CC$ by setting $b_{n}(t) = R \widehat{m}(R (t + n))$. Then $b = \sum_n b_{n}$. Moreover,
    \begin{align} \label{translationlpcalculation}
    \begin{split}
        &\left( \sum_{n \neq 0} \left[ \langle R n \rangle^{\alpha(p)} \| b_{n} \|_{L^p(I_0)} \right]^p \right)^{1/p}\\
        &\quad\quad\quad \sim \left( \int_{|t| \geq 1/2} \left[ \langle R t \rangle^{\alpha(p)} |R \widehat{m}(R t)| \right]^p \right)^{1/p}\\
        &\quad\quad\quad = R^{1/p'} \left( \int_{|t| \geq R/2} \left[ |t|^{\alpha(p)} \widehat{m}(t) \right]^p \right)^{1/p} \leq R^{1/p'} C_p(m).
    \end{split}
    \end{align}
    Write $b_{t_0}^I = \sum_n b_{t_0,n}^I$ and $b^{II} = \sum_n b_{n}^{II}$, where $b_{t_0,n}^I = \chi_{t_0} b_n$ and $b_n^{II} = \mathbb{I}_{I_0 \smallsetminus [-\varepsilon, \varepsilon] } b_{n}$. Then
    \begin{align} \label{zerobj0iicalculation}
    \begin{split}
        \| b_{0}^{II} \|_{L^p(I_0)} &= \left( \int_{\varepsilon \leq |t| \leq 1/2} |R \widehat{m}(R t)|^p \right)^{1/p}\\
        &= R^{1/p'} \left( \int_{R \varepsilon \leq |t| \leq R/2} |\widehat{m}(t)|^p \right)^{1/p} \lesssim R^{1/p' - \alpha(p)} C_p(m).
    \end{split}
    \end{align}
    Using \eqref{translationlpcalculation}, \eqref{zerobj0iicalculation}, and H\"{o}lder's inequality, we conclude that
    \begin{align}
    \begin{split}
        \|  b^{II} \|_{L^p(I_0)} &\leq \sum\nolimits_n \| b_{n}^{II} \|_{L^p(I_0)}\\
        &\leq \| b_{0}^{II} \|_{L^p(I_0)} + \sum\nolimits_{n \neq 0} \left[ |R n|^{\alpha(p)} \| b_{n}^{II} \|_{L^p(I_0)} \right] \frac{1}{|R n|^{\alpha(p)}}\\
        &\leq \| b_{0}^{II} \|_{L^p(I_0)} + R^{-\alpha(p)} \left( \sum\nolimits_{n \neq 0} \left[ |R n|^{\alpha(p)} \| b_{n} \|_{L^p(I_0)} \right]^p \right)^{1/p}\\ 
        &\lesssim R^{1/p' - \alpha(p)} C_p(m).
    \end{split}
    \end{align}
    A similar calculation shows that
    \begin{align} \label{eacht0bjcalculation}
    \begin{split}
        \| b_{t_0}^I \|_{L^p(I_0)} &\leq \sum\nolimits_n \| b_{t_0,n}^I \|_{L^p(I_0)}\\
        &= \| b_{t_0,0}^I \|_{L^p(I_0)} + \sum\nolimits_{n \neq 0} \| b_{t_0,n}^I \|_{L^p(I_0)}\\
        &\lesssim \| b_{t_0,0}^I \|_{L^p(I_0)} + R^{-\alpha(p)} \Big( \sum\nolimits_{n \neq 0} |R n|^{\alpha(p)} \| b_{t_0,n}^I \|_{L^p(I_0)}^p \Big)^{1/p}\\
        &\lesssim \| b_{t_0,0}^I \|_{L^p(I_0)} + \Big( \sum\nolimits_{n \neq 0} |R n|^{\alpha(p)} \| b_{t_0,n}^I \|_{L^p(I_0)}^p \Big)^{1/p}.
    \end{split}
    \end{align}
    Using \eqref{eacht0bjcalculation}, we calculate that
    \begin{align} \label{bjt0Icalculation}
    \begin{split}
        &\left( \sum_{t_0 \in \mathcal{T}_R} \left[ \| b_{t_0}^I \|_{L^p(I_0)} \langle R t_0 \rangle^{\alpha(p)} \right]^p \right)^{1/p}\\
        &\quad\quad \lesssim \left( \sum_{t_0 \in \mathcal{T}_R} \left[ \| b_{t_0,0}^I \|_{L^p(I_0)} \langle R t_0 \rangle^{\alpha(p)} \right]^p + \sum_{n \neq 0} \left[ |R n|^{\alpha(p)} \| b_{t_0,n}^I \|_{L^p(I_0)} \right]^p \right)^{1/p}\\
        &\quad\quad \lesssim \left( \int_{\RR} \left[ \langle R t \rangle^{\alpha(p)} R \widehat{m}(R t) \right]^p dt \right)^{1/p}\\
        &\quad\quad \lesssim R^{1/p'} C_p(m).
    \end{split}
    \end{align}
    Since each function $b_{t_0}^I$ is supported on a length $1/R$ interval, we have
    \begin{equation}
        \| b_{t_0}^I \|_{L^1(I_0)} \lesssim R^{-1/p'} \| b_{t_0}^I \|_{L^p(I_0)},
    \end{equation}
    and substituting this inequality into \eqref{bjt0Icalculation} completes the proof.
\end{proof}

\begin{proof} [Proof of Lemma \ref{pseudodifferentialCoordinateLemma}]
    For each $\alpha$, given our choice of $\varepsilon_M$, the Lax-H\"{o}rmander Parametrix construction (see Theorem 4.1.2 of \cite{Sogge}) guarantees that we can find operators $\tilde{W}_\alpha(t)$ and $\tilde{R}_\alpha(t)$ for $|t| \leq \varepsilon_M$, such that for $u \in L^1(M)$ with $\text{supp}(u) \subset V_\alpha^*$,
    \begin{equation}
        e^{2 \pi i t P} u = \tilde{W}_\alpha(t) u + \tilde{R}_\alpha(t) u,
    \end{equation}
    where $\tilde{R}_\alpha(t)$ has a smooth kernel, and the kernel of $\tilde{W}_\alpha(t)$ is given in coordinates by
    \begin{equation}
        \tilde{W}_\alpha(t)(x,y) = \int s_0(t,x,y,\xi) e^{2 \pi i [ \phi(x,y,\xi) + t p(y,\xi) ]}\; d\xi,
    \end{equation}
    for an order zero symbol $s_0$ with
    \begin{equation}
        \text{supp}_{x,y,\xi}(s_0) \subset \{ (x,y,\xi) \in U_\alpha \times V_\alpha^* \times \RR^d : d_M(x,y) \leq 1.01\varepsilon_M\ \text{and}\ |\xi| \leq 1 \},
    \end{equation}
    and an order one symbol $\phi$, homogeneous in $\xi$ of order one, solving the Eikonal equation and vanishing for $x \in \Sigma_\alpha(y,\xi)$ as required by the lemma. The only difference here compared to Theorem 4.1.2 of \cite{Sogge} is that in that construction the function $\phi$ there is chosen to vanish for $x \in \tilde{\Sigma}_\alpha(y,\xi)$, where $\tilde{\Sigma}_\alpha(y,\xi) = \{ x : \xi \cdot (x - y) = 0 \}$. The only property of this choice that is used in the proof is that the perpendicular vector to $\tilde{\Sigma}_\alpha(y,\xi)$ at $y$ is $\xi$, and this is also true of the hypersurfaces $\Sigma_\alpha(y,\xi)$ that we have specified, so that there is no problem making this modification.

    In the remainder of the proof it will be convenient to fix a orthonormal basis $\{ e_k \}$ of eigenfunctions for $P$, such that $\Delta e_k = \lambda_k e_k$ for a non-decreasing sequence $\{ \lambda_k \}$. We fix $u \in L^1(M)$ with $\text{supp}(u) \subset V_\alpha^*$ and $\| u \|_{L^1(M)} \leq 1$.

    We begin by mollifying the functions $Q$. We proceed here with a similar approach to Theorem 4.3.1 of \cite{Sogge}. We fix $\rho \in C_c^\infty(\RR)$ equal to one in a neighborhood of the origin and with $\rho(t) = 0$ for $|t| \geq \varepsilon_M / 2$. We write
    \begin{equation}
    \begin{split}
        Q &= \int R \widehat{q}(R t) e^{2 \pi i t P}\; dt\\
        &= \int R \widehat{q}(Rt) \Big\{ \rho(t) \tilde{W}(t) + \rho(t) \tilde{R}(t) + (1 - \rho(t)) e^{2 \pi i t P} \Big\}\; dt\\
        &= Q_I + Q_{II} + Q_{III}.
    \end{split}
    \end{equation}
    The rapid decay of $\widehat{q}$ implies that the function $\psi(t) = R \widehat{q}(Rt) (1 - \rho(t))$ satisfies $\| \partial_t^N \psi \|_{L^1(\RR)} \lesssim_M R^{-M}$, and so
    \begin{equation} \label{psidecaybound}
        |\widehat{\psi}(\lambda)| \lesssim_{N,M} R^{-M} \lambda^{-N}.
    \end{equation}
    But since $Q_{III} = \widehat{\psi}(-P)$, we can write the kernel of $Q_{III}$ as
    \begin{equation}
        Q_{III}(x,y) = \sum\nolimits_\lambda \widehat{\psi}(-\lambda_k) e_k(x) \overline{e_k(y)},
    \end{equation}
    Sobolev embedding and \eqref{psidecaybound} imply that $|Q_{III}(x,y)| \lesssim_N R^{-N}$ and thus
    \begin{equation} \label{QThreeBound}
        \| Q_{III} u \|_{L^\infty(M)} \lesssim_N R^{-N}.
    \end{equation}
    Integration by parts, using the fact that $q$ vanishes near the origin, yields that
    \begin{equation}
        \left| \int R \widehat{q}(Rt) \rho(t) \tilde{R}(t,x,y) \right| \lesssim_N R^{-N},
    \end{equation}
    and thus
    \begin{equation} \label{QTwoBound}
        \| Q_{II} u \|_{L^\infty(M)} \lesssim_N R^{-N}.
    \end{equation}
    Now we expand
    \begin{equation}
        Q_I = \iint R \widehat{q}(Rt) \rho(t) s_0(t,x,y,\xi) e^{2 \pi i [ \phi(x,y,\xi) + t p(y,\xi) ]}\; d\xi\; dt.
    \end{equation}
    We perform a Fourier series expansion, writing
    \begin{equation} c_n(x,y,\xi) = \int \rho(t) s_0(t,x,y,\xi) e^{-2 \pi i n t}\; dt. \end{equation}
    Then the symbol estimates for $s_0$, and the compact support of $\rho$ imply that
    \begin{equation} |\partial_{x,y}^\alpha \partial_\xi^\beta c_n(x,y,\xi)| \lesssim_{\alpha,\beta,N} |n|^{-N} \langle \xi \rangle^{-\beta}. \end{equation}
    Using Fourier inversion we can write
    \begin{equation}
    \begin{split}
        Q_I(x,y) &= \iint \sum\nolimits_n R \widehat{q}(Rt) c_n(x,y,\xi) e^{2 \pi i [ \phi(x,y,\xi) + t [ n + p(y,\xi) ] ]}\; d\xi\; dt\\
        &= \int \sum\nolimits_n q \Big( \big( n + p(y,\xi) \big) / R \Big) c_n(x,y,\xi) e^{2 \pi i \phi(x,y,\xi)}\; d\xi\\
        &= \int \tilde{\sigma}_\alpha(x,y,\xi) e^{2 \pi i \phi(x,y,\xi)}\; d\xi, 
    \end{split}
    \end{equation}
    where
    \begin{equation} \tilde{\sigma}_\alpha(x,y,\xi) = \sum_{n \in \ZZ} q \left( \frac{n + p(y,\xi)}{R} \right) c_n(x,y,\xi). \end{equation}
    The $n$th term of this sum is supported on $R/4 - n \leq p(y,\xi) \leq 4R - n$, so in particular, if $n > 4R$ then the term vanishes. For $n \leq 4R$, we have estimates of the form
    %
    %
    \begin{equation} \left| \partial_{x,y}^\alpha \partial_\xi^\beta \left\{ q \left( \frac{n + p(y,\xi)}{R} \right) c_n(x,y,\xi) \right\} \right| \lesssim_{\alpha,\beta,N} |n|^{-N}. \end{equation}
    and for $-4R \leq n \leq R/8$,
    \begin{equation}
    \begin{split}
        \left| \partial_{x,y}^\alpha \partial_\xi^\beta \left\{ q \left( \frac{n + p(y,\xi)}{R} \right) c_n(x,y,\xi) \right\} \right| &\lesssim_{\alpha,\beta,N} |n|^{-N} R^{-\beta}.
    \end{split}
    \end{equation}
    But this means that if we define
    \begin{equation} \sigma_\alpha(x,y,\xi) = \sum\nolimits_{-4R \leq n \leq R/8} q \left( \frac{n + p(y,\xi)}{R} \right) c_n(x,y,\xi). \end{equation}
    and define
    \begin{equation} Q_\alpha(x,y) = \int \sigma_\alpha(x,y,\xi) e^{2 \pi i \phi(x,y,\xi)}\; d\xi \end{equation}
    then
    \begin{equation} \Big|\partial_{x,y}^\alpha \partial_\xi^\beta \big\{ \tilde{\sigma}_\alpha - \sigma \big\}(x,y,\xi) \Big| \lesssim_{\alpha,\beta,N,M} R^{-N} \langle \xi \rangle^{-M}, \end{equation}
    and so
    \begin{equation} \label{QalphaApproximation}
        \| ( Q_I - Q_\alpha ) u \|_{L^\infty(M)} \lesssim_N R^{-N}.
    \end{equation}
    Combining \eqref{QThreeBound}, \eqref{QTwoBound}, and \eqref{QalphaApproximation}, we conclude that
    \begin{equation} \label{QApproximationTheorem}
        \| (Q - Q_\alpha) u \|_{L^\infty(M)} \lesssim_N R^{-N}.
    \end{equation}
    Since $\sigma_\alpha$ is supported on $|\xi| \sim R$, we have verified the required properties of $Q_\alpha$.

    Using the bounds on $Q - Q_\alpha$ obtained above, we see that
    \begin{equation} \big\| [(Q \circ e^{2 \pi i t P} \circ Q) - (Q_\alpha \circ e^{2 \pi i t P} \circ Q_\alpha)] \{ u \} \big\|_{L^\infty(M)} \lesssim_N R^{-N}. \end{equation}
    Since $\tilde{R}_\alpha(t)$ has a smooth kernel, we also see that
    \begin{equation}
    \begin{split}
        &\big\| (Q_\alpha \circ (e^{2 \pi i t P} - \tilde{W}_\alpha(t)) \circ Q_\alpha) \{ u \} \big\|_{L^\infty(M)}\\
        &\quad\quad = \big\| (Q_\alpha \circ \tilde{R}_\alpha(t) \circ Q_\alpha) \{ u \} \big\|_{L^\infty(M)} \lesssim_N R^{-N}.
    \end{split}
    \end{equation}
    We now write
    \begin{equation}
    \begin{split}
        &(Q_\alpha \circ \tilde{W}_\alpha(T))(x,y)\\
        &\quad = \int \sigma_\alpha(x,\xi) s_0(t,z,y,\eta) e^{2 \pi i [\phi(x,z,\xi) + \phi(z,y,\eta) + t p(y,\eta) ]}\; d\xi\; d\eta\; dz.
    \end{split}
    \end{equation}
    The phase of this equation has gradient in the $z$ variable with magnitude $\gtrsim R$ for $|\eta| \ll R$, and $\gtrsim R |\xi|$ if $|\eta| \gg R$. Thus, if we define $s(t,x,y,\xi) = s_0(t,x,y,\xi) \chi(\xi / R)$ where $\text{supp}(\chi) \subset [1/8,8]$, and then define
    \begin{equation} W_\alpha(t)(x,y) = \int s(t,x,y,\xi) e^{2 \pi i [ \phi(x,y,\xi) + t p(y,\xi) ]}\; d\xi, \end{equation}
    then we may integrate by parts in the $z$ variable to conclude that
    \begin{equation} \Big|\big(Q_R \circ (\tilde{W}_\alpha(t) - W_\alpha(t)) \big)(x,y) \Big| \lesssim_N R^{-N}, \end{equation}
    and thus
    \begin{equation} \| (Q_R \circ (\tilde{W}_\alpha(t) - W_\alpha(t)) \circ Q_R) u \|_{L^\infty(M)} \lesssim R^{-N}. \end{equation}
    This proves the required estimates for the operators $W_\alpha$.
\end{proof}

\pagebreak[4]

\bibliographystyle{amsplain}
\bibliography{MultipliersOfLaplacianOnSd}

\end{document}